\newif\ifARXIV
\ARXIVtrue

\newif\ifMP
\MPfalse

\newif\ifSIAM
\SIAMfalse

\ifMP
\documentclass[sn-mathphys-num]{sn-jnl}
\usepackage{graphicx}%
\usepackage{multirow}%
\usepackage{amsmath,amssymb,amsfonts}%
\usepackage{amsthm}%
\usepackage{mathrsfs}%
\usepackage[title]{appendix}%
\usepackage{textcomp}%
\usepackage{manyfoot}%
\usepackage{booktabs}%
\usepackage{algpseudocode}%
\usepackage{listings}%
\usepackage{graphicx,fancyhdr,url}
\usepackage{amsthm}
\usepackage{geometry}                		
\usepackage{orcidlink}   
\fi

\newcommand{\PaperTitle}[0]{An adaptive interior-point method with backtracking line search for convex constrained optimization}
\newcommand{\MyThanks}[0]{Department of Industrial Engineering, University of Pittsburgh. Email: \{fah33, ohinder\}@pitt.edu.}

\ifARXIV
\documentclass{article}
\author{Fadi Hamad, Oliver Hinder\thanks{\MyThanks}}
\fi

\ifSIAM
\documentclass[review,hidelinks,onefignum,onetabnum]{siamart250211}
\headers{An adaptive interior-point method with backtracking line search for convex constrained optimization}{Fadi Hamad and Oliver Hinder}
\title{An adaptive interior-point method with backtracking line search for convex constrained optimization \funding{This work was funded by AFOSR grant \#FA9550-23-1-0242}}

\author{Fadi Hamad\thanks{Department of Industrial Engineering, University of Pittsburgh, 
\email{fah33@pitt.edu}.}
\and Oliver Hinder\thanks{Department of Industrial Engineering, University of Pittsburgh,
\email{ohinder@pitt.edu}.}}
\fi

\usepackage{geometry}                		
\usepackage{graphicx}	
\ifARXIV
\usepackage{subcaption}
\fi
\usepackage{amssymb}
\usepackage{amsthm}
\usepackage{tabularx}
\usepackage{algorithm}
\usepackage{algpseudocode}

\usepackage{enumerate}
\usepackage{longtable}
\usepackage{listings}
\usepackage{booktabs}
\usepackage{tablefootnote}
\usepackage{multirow}
\ifARXIV
\usepackage[hidelinks]{hyperref}
\fi
\newcommand{\eye}{\mbox{\bf I}}

\usepackage{physics}

\newtheorem{theorem}{Theorem}
\newtheorem{lemma}{Lemma}
\newtheorem{claim}{Claim}
\newtheorem{fact}{Fact}
\newtheorem{proposition}{Proposition}

\newtheorem{remark}{Remark}
\newtheorem{definition}{Definition}
\newtheorem{condition}{Condition}

\usepackage{accents}
\newcommand{\ubar}[1]{\underaccent{\bar}{#1}}
\usepackage{enumitem}

\newcommand{\undereq}[1]{\underset{#1}{=}}
\newcommand{\underle}[1]{\underset{#1}{\le}}
\newcommand{\underge}[1]{\underset{#1}{\ge}}

\newcommand{\diam}{\mathbf{diam}}

\newcommand{\N}{\mathbf{N}}
\newcommand{\R}{\mathbf{R}}

\newcommand{\argmin}{\mathop{\rm argmin}}

\def\citet{\cite}
\def\citep{\cite}

\usepackage{cleveref}

\crefname{fact}{fact}{facts}
\Crefname{fact}{Fact}{Facts}
\crefname{condition}{condition}{conditions}
\Crefname{condition}{Condition}{Conditions}

\newtheorem{assumption}{Assumption}

\newcommand{\barrier}[1][]{%
\ifthenelse{\isempty{#1}}
{\psi_{\mu}}
{\psi_{\mu_{#1}}}
}
\newcommand{\NumCon}{q}
\newcommand{\cons}{a}
\newcommand{\Lag}{\mathcal{L}}
\newcommand{\ones}{\mathbf 1}
\newcommand{\zeros}{\mathbf 0}
\newcommand{\reals}{{\mbox{\bf R}}}
\newcommand{\Convex}{CONVEX}
\newcommand{\dir}{d^}
\newcommand{\X}{\mathcal{X}}
\newcommand{\LipGrad}{L_G}
\newcommand{\LipHess}{L_H}
\newcommand{\NumVar}{n}
\newcommand{\StoppingTolerance}{\bar{\mu}}

\newcommand{\ConSet}[0]{[ \NumCon ]}
\newcommand{\SufficientConstant}{C}
\newcommand{\diag}{\mathop{\bf diag}}
\newcommand{\interior}[1]{\mathbf{interior}(#1)}

\newcommand{\backtrackingFactor}{\gamma}
\usepackage{amsmath,amsthm}
\newcommand{\alphaInd}[1]{\ubar{\alpha}_{#1}}
\usepackage{color}
\ifARXIV
\usepackage[numbers,square,sort&compress]{natbib}
\fi
\newcommand{\Mk}{\mathcal{M}_k}
\newcommand{\mk}{M_k}

\newcommand{\ykPlusOne}[0]{\sigma_k}

\newcommand{\ColdStartScale}{\rho_{c}}
\newcommand{\RegularizationShrink}{\rho_{s}}
\newcommand{\StartingTargetReduction}{\tau}

\newcommand{\DeltaInit}{\tilde{\delta}}
\newcommand{\DeltaMid}{\delta_{\mathbf{mid}}}
\newcommand{\DirDelta}[1][\delta]{d(#1)}

\newcommand{\SideA}[1][i]{a_{#1}}
\newcommand{\SideB}[1][i]{b_{#1}}

\algdef{SE}[INDENT]{Indent}{EndIndent}{}{}
\algtext*{Indent}     
\algtext*{EndIndent}  

\newcommand{\edit}[1]{{#1}}
\begin{document}
	
\newcommand{\abstractText}[0]{Interior-point methods (IPMs) are widely adopted due to their high practical efficiency in solving linear, convex, and nonconvex optimization problems. 
	For convex optimization, this performance is theoretically well-supported: there are strong complexity guarantees for self-concordant barrier setups \cite{nesterov1994interior}, which cover linear and conic optimization. 
	On the other hand, unconstrained convex optimization is well-studied. 
	However, there is limited analysis of constrained convex optimization methods without the self-concordance assumption. 
	We develop and analyze a regularized Newton method with line search applied to the log barrier function in 
	the setting that the objective and constraints are thrice differentiable and have Lipschitz continuous first and second derivatives. 
	Starting from a strictly feasible point, our method finds an $\epsilon$-approximately optimal solution in $\tilde{O}(\epsilon^{-2/3})$ iterations.
}

\ifMP
\else
\title{\PaperTitle}
\maketitle
\begin{abstract}
	\abstractText
\end{abstract}
\fi

\section{Introduction}

This paper studies the constrained convex optimization problem
\begin{flalign}\label{eq:optimization-constrained-problem}
	\min_{x \in \X }{~f(x)} \text{ where } \X := \{ x \in \reals^{\NumVar} : \cons(x) \ge \zeros \}
\end{flalign}
where $f: \R^{\NumVar} \rightarrow \R$ is convex, and $\cons: \R^{\NumVar} \rightarrow \R^{\NumCon}$ is concave.
We assume that both functions are thrice differentiable with Lipschitz continuous first and second derivatives. 
Additionally, we assume Slater's condition holds, i.e., there exists a strictly feasible point $x_1$ such that $\cons(x_1) > \zeros$, which we use as the starting point of our algorithm. 
Under these assumptions, we aim to find an approximately optimal solution, i.e., a point $\tilde{x}$ such that
\[
f(\tilde{x}) \le \epsilon + \min_{x \in \X }{~f(x)} 
\]
for some small $\epsilon > 0$.
To solve this problem, we move the constraints into the objective by using a log barrier function:
\begin{flalign}\label{barrier-problem}
	\barrier(x) := f(x) - \mu \sum_{i=1}^{\NumCon} \log(\cons_i(x))
\end{flalign} 
where the barrier parameter $\mu$ is positive. 
This transforms the constrained optimization problem~\eqref{eq:optimization-constrained-problem} into an unconstrained problem~\eqref{barrier-problem}, allowing us to use methods from unconstrained optimization such as damped Newton's method \cite{nesterov1994interior}.
More specifically, we employ a regularized Newton method with line search, i.e.,
\[
x_{k+1} = x_{k} + \alpha_k d^{x}_k
\]
where $\delta_k$ is a carefully selected regularization parameter,
$d^{x}_k :=  -(\grad^2 \barrier(x_k) + \delta_k \eye)^{-1} \grad \barrier(x_k)$ is the regularized Newton direction,
and $\alpha_k \in [0,1]$ is the step size, e.g., selected via a backtracking line search.
This is a popular practical strategy for solving constrained nonlinear optimization problems \cite{vanderbei1999loqo,wachter2006implementation} (also, see \citet[Algorithm 3.2]{nocedal2006numerical}).
However, there are no known worst-case complexity bounds for this type of method.

A natural approach would be to apply standard trust-region or cubic regularization methods to solve this problem \cite{nesterov2006cubic,cartis2011adaptiveII,hamad2022consistently,curtis2017trust}. Unfortunately, these methods 
require that the Hessian is Lipschitz continuous. However, the Hessian of the log barrier is not
Lipschitz continuous. Moreover, even if we 
only require that the Hessian of the log barrier is Lipschitz continuous on the initial level set, $\{ x \in \X : \barrier(x) \le \barrier(x_1) \}$,
the corresponding Lipschitz constant can grow  
exponentially with the inverse of the barrier parameter $\mu$, leading to extraordinarily large
worst-case iteration bounds \cite[Section 7.2.]{hinder2019principled}.

\paragraph{Our contributions:}

\begin{enumerate}
\item We provide the first worst-case bound for regularized Newton with line search on log barriers for 
general convex constrained optimization problems. To prove this result we develop a new proof technique
 based by carefully constructing a dynamic lower bound on our method's step size.
\item We establish an $\tilde{O}(\epsilon^{-2/3})$ bound on the number of iterations of our method.
\item Despite its simplicity, our method shows strong potential in preliminary comparisons with 
IPOPT \cite{wachter2006implementation}, a state-of-the-art interior point method.
\end{enumerate}

Our techniques build on ideas from \citet{hinder2024worst}, which provides an algorithm that, while impractical because it requires knowledge of problem parameters such as the Lipschitz constant, provides an $O(\mu^{-7/4})$ iteration bound for log barrier methods on general nonconvex constrained optimization problems to find $\mu$-approximate KKT points.

\subsection{Related work}\label{sec:related-work}

\paragraph{Second-order methods for unconstrained convex optimization with Lipschitz assumptions.}
While this paper is unique because it develops an implementable and adaptive second-order method for \emph{constrained}
convex optimization with Lipschitz continuous derivatives, there is a substantial line of
work for \emph{unconstrained} convex optimization under such
assumptions.
 The cubically regularized Newton method of \citet{nesterov2006cubic}
established the first global complexity bounds in this setting under Lipschitz continuity
of the Hessian, achieving an $O(1/k^2)$ rate for convex problems, which was subsequently
accelerated to $O(1/k^3)$ by \citet{nesterov2008accelerating}. \citet{monteiro2013accelerated}
proposed the accelerated hybrid proximal extragradient (A-HPE) framework and a corresponding
accelerated Newton proximal extragradient (A-NPE) method with an improved
$\tilde{O}(1/k^{7/2})$ rate, which was later shown to be optimal up to logarithmic factors
by the lower bounds of \citet{agarwal2018lower} and \citet{arjevani2019oracle}.
Lipschitz continuity of higher derivatives was exploited by
\citet{nesterov2021superfast} to obtain an $O(1/k^4)$ convergence rate. More recently,
\citet{mishchenko2023regularized} develop an adaptive regularized
Newton method for convex optimization with excellent practical performance.

\paragraph{Worst-case guarantees for general constrained convex optimization via variational inequalities.}
One technique for solving general 
convex optimization problems is to reformulate the problem in terms of the optimality conditions and use variational inequality methods to solve them \cite{rockafellar1976monotone}.
Under the assumption that the gradient is Lipschitz, mirror-prox uses a first-order oracle to obtain an $O(\epsilon^{-1})$ convergence rate \cite{nemirovski2004prox}.
Under the assumption of Lipschitz continuous Hessians, \citet{monteiro2012iteration} introduced a Newton proximal extragradient method and obtained a $\tilde{O}(\epsilon^{-2/3})$ iteration complexity for monotone variational inequalities. 
This $\tilde{O}(\epsilon^{-2/3})$ rate has since been recovered by a number of more recent methods \cite{bullins2022higher,jiang2025generalized,lin2024perseus,adil2022optimal}; \citet{adil2022optimal} also provide matching upper and lower bounds showing that $O(\epsilon^{-2/3})$ is the optimal oracle complexity. 
A key caveat is that these are \emph{oracle} rates. In particular, each iteration involves finding an approximate solution of a variational inequality derived from a second-order Taylor model. 
Solving these subproblems reduces to a few linear system solves in the unconstrained setting but
once even simple constraints such as the nonnegative orthant are imposed, this subproblem becomes more expensive. For example, the subproblem of \citet[Definition~3.1]{monteiro2012iteration} reduces to a monotone linear complementarity problem, which can be solved by interior point methods \cite{zhang1994convergence,simantiraki1997infeasible} at higher expense than a linear system solve.
Our work obtains a similar $\tilde{O}(\epsilon^{-2/3})$ rate\footnote{These bounds are not directly comparable because (i) of the aforementioned subproblem solver, (ii) our final bound (\Cref{thm:constrained-case}) includes a dependence on the number of constraints, $\NumCon$, and (iii) we assume the problem has a strict relative interior.} for general constrained convex problems but with cheaper and more practical iterations: each iteration of our method only requires a logarithmic number of linear system solves and evaluations of $\Lag$, $\grad \Lag$, and $\grad^2 \Lag$, where $\Lag(x,y) := f(x) - y^\top \cons(x)$ is the Lagrangian. 

\paragraph{Interior-point methods} 
Interior-point methods (IPMs) are a very popular second-order method to solve constrained optimization problems, including in linear, conic, convex, and nonconvex settings \cite{andersen1998computational, lustig1992implementing, sturm1999using,
toh1999sdpt3, byrd2000trust, conn2000primal, byrd2006knitro, vanaret2025implementing, wachter2006implementation}.
The main theoretical complexity result in the convex setting for IPMs states that they require $O(\sqrt{c}\log(1/\epsilon))$ iterations to find an $\epsilon$-global minimum \cite{nesterov1994interior}, where $c$ is the self-concordance parameter of the barrier function and $\epsilon$ stands for the desired accuracy or duality gap. 
However, this theoretical foundation is restrictive: for general nonlinear convex problems
the log barrier is not self-concordant and, while the universal barrier exists
for arbitrary convex bodies, it is computationally intractable \cite{nesterov1994interior,pmlr-v40-Bubeck15b}.

\paragraph{Quasi-self-concordance} Quasi-self-concordance is a variant of classical self-concordance introduced by \citet{bach2010self} in the context of logistic regression. Within this class, \citet{carmon2020acceleration} obtain accelerated rates via a ball optimization oracle, and \citet{doikov2025minimizing} show that gradient-regularized Newton methods can also obtain similar rates. 
While this structure is well-suited to certain machine learning problems, it cannot accommodate barrier functions such as $-\log(x)$, whose curvature diverges at the boundary, and is therefore not applicable to interior point methods for general constrained problems.

\paragraph{Adaptive second-order methods for nonconvex optimization} 
While this paper is focused on developing an adaptive method for convex optimization, there is a large body of work 
studying adaptive methods for nonconvex optimization both in the unconstrained \cite{cartis2011adaptiveII,curtis2017trust,curtis2021trust,hamad2022consistently,hamad2024simple} 
and constrained settings \cite{birgin2016evaluation, cartis2011evaluation,cartis2012adaptive}.
However, \citet{hinder2024worst} is the only paper studying log barrier methods with general (nonconvex) 
constraints, but it requires knowledge of problem parameters such as the Lipschitz constant, making it 
nonadaptive and impractical.

\subsection{Outline and notation}

\paragraph{Outline} The paper is structured as follows. In \Cref{sec:our-adaptive-line-search-IPM}, 
we introduce and analyze our adaptive line search method for minimizing the 
log barrier with a fixed barrier parameter showing that it can find 
a $\mu$-approximate stationary interior point in $O(\mu^{-5/3})$ iterations;
proving this result involves developing a novel proof technique that reasons about the minimum step size thresholds 
relative to the predicted reduction in the log barrier for a unit regularized Newton step.
In \Cref{sec:optimality-guarantees-convexity-regularity-condition}
we apply a standard scheme for annealing the barrier parameter to 
our method from \Cref{sec:our-adaptive-line-search-IPM}, yielding 
our main result, which collapses our $O(\mu^{-5/3})$ iteration bound into 
a $O(\epsilon^{-2/3})$ iteration bound for finding an $\epsilon$-approximately optimal solution.
In \Cref{sec:numerical-results}, we provide the numerical results.
Finally, \Cref{sec:discussion} provides a discussion.

\paragraph{Notation} Let $\N$ be the set of natural numbers (starting from one), $\eye$ be the identity matrix, $\R$ be the set of real numbers, and $\R_{+}$ be the set of nonnegative real numbers. Let $\diag(v)$ be a diagonal matrix with entries composed of the vector $v$. Let $\Convex\{u, v\} = \{\theta u + (1-\theta) v : \theta \in [0, 1]\}$. Let $\barrier^{\star} := \inf_{x \in \X} \barrier(x)$, and $f_{\star} := \inf_{x \in \X} f(x)$. Unless otherwise
specified, $\log(\cdot)$ is the natural logarithm. 
Let $\interior{\X} := \{ x \in \R^{\NumVar} : \cons(x) > \zeros \}$.
Let $\diam(\X) = \sup_{u,v \in \X} \| u - v \|_2$.
\section{A log barrier method with fixed barrier parameter}\label{sec:our-adaptive-line-search-IPM}

In this section, we detail our adaptive line search log barrier method (\Cref{alg:IPM}). 
To compute the next iterate $x_{k+1}$, we first carefully select a regularized Newton search direction $\dir{x}_k$ 
(Condition~\ref{condition:search-dir}), and then determine a step size $\alpha_k \in (0, 1]$ via a line search
(Condition~\ref{step-size-condition}).

To generate the search direction, we consider the second-order Taylor series
approximation of the function at the current iterate:
\begin{flalign}\label{eq:newton-direction-IPM-unconstrained}
\Mk(d) := \grad \barrier(x_k)^\top d + \frac{1}{2} d^\top \grad^2 \barrier(x_k) d.
\end{flalign}
When generating our search direction, instead of taking a full Newton step and minimizing \Cref{eq:newton-direction-IPM-unconstrained}, we take a regularized Newton step:
\begin{flalign*}
	\dir{x}_k := -(\grad^2 \barrier(x_k) + \delta_k \eye)^{-1} \grad \barrier(x_k). 
\end{flalign*}
We require our search direction $\dir{x}_k$ to satisfy Condition~\ref{condition:search-dir}.
For brevity, we denote $\mk = \Mk(\dir{x}_k)$. 

\begin{condition}\label{condition:search-dir} 	
	Let $\eta_1$ and $\eta_2$ be problem-independent constants with $\eta_1 \in (0, 1/2)$ and $\eta_2 \in (\eta_1, 1/2)$.
	Let $\varepsilon_k > 0$ be a sequence of scalars with $\| \grad \barrier(x_k) \|_2 > \varepsilon_k$.
	The search direction $\dir{x}_k$ and regularization parameter $\delta_k \ge 0$ satisfy
	\begin{subequations}
		\begin{flalign}
			(\grad^2 \barrier(x_k) + \delta_k \eye) \dir{x}_k =& -\grad \barrier(x_k) \label{eq:search-direction} \\
			\eta_1 \le & ~\frac{-\mk}{\varepsilon_k \|\dir{x}_k\|_2} \label{eq:direction-condition-eta1} \\
			\text{either }\barrier(x_k + \dir{x}_k) > \barrier(x_k) + \SufficientConstant \cdot \mk , ~\text{ or }~ &\frac{\| \dir{x}_k \|_2 \delta_k}{2 \varepsilon_k} \le \eta_2 \label{eq:direction-condition-eta2}
		\end{flalign}	
	\end{subequations}
	where $\varepsilon_k := \mu \sqrt{1 + \| y_k \|_1}$.
\end{condition}

\noindent In \Cref{appendix-satisfy-condition-phi-unconstrained} we show it is always possible to satisfy 
Condition~\ref{condition:search-dir} with at most 
\[
O\left( 1
+\log_2\log_2 
\left(\frac{\| \grad^2 \barrier(x_k) \|_2}
{\DeltaInit} +
\frac{\DeltaInit(\barrier(x_k) - \barrier^\star)}
{\SufficientConstant \varepsilon_k^2}
\right) \right)
\]
factorizations and function evaluations where $\DeltaInit$ is some initial estimate of the regularization parameter.
The key idea is that it suffices to find a regularizer $\delta$ such that
\begin{flalign}\label{eq:bisection-scheme}
	\phi_k(\delta)  \in [\eta_1, \eta_2] \text{ where } \phi_k(\delta) := \varepsilon_k^{-1} \delta \| (\grad^2 \barrier(x_k) + \delta \eye)^{-1}  \grad \barrier(x_k) \| .
\end{flalign}
We find such a regularizer by applying a bisection search to $\phi_k$ on a log-scale. 
To find a search interval for the bisection routine we devise a simple doubling scheme on a log-scale that leverages several important 
facts we can prove about $\phi_k$, for all $\delta > 0$: $(i)$ for $\delta' \ge \delta$, $\phi_k(\delta) \le \phi_k(\delta')$
and $\delta \phi_k(\delta') \le \delta' \phi_k(\delta)$, 
$(ii)$ for $\delta \ge \frac{\eta_1}{1 - \eta_1} \| \grad^2 \barrier(x_k) \|_2$ we have 
$\phi_k(\delta) \ge \eta_1$, $(iii)$ for $\delta \le \eta_2^2 \SufficientConstant \varepsilon_k^2 (\barrier(x_k) - \barrier^\star)^{-1}$ 
we have $\phi_k(\delta) \le \eta_2$ or
$\barrier(x_k + \DirDelta) > \barrier(x_k) + \SufficientConstant \Mk(\DirDelta)$. These facts allow us not only to efficiently find an initial interval
but that the endpoint of the interval will be sufficiently close for the bisection search to be efficient.

Given a search direction satisfying Condition~\ref{condition:search-dir}, we now find a step size such that the sufficient decrease condition holds 
but for larger step size it fails (Condition~\ref{step-size-condition}). There are many line search methods that one could use 
to satisfy such a condition but the particular method we use in our implementation is a variant of the standard 
backtracking method. See Appendix~\ref{appendix-line-search-routine} for details and guarantees on the number of function 
evaluations of this routine.

\begin{condition}[Step size condition]\label{step-size-condition}
	Let $\SufficientConstant \in (0, 1/2)$ and $\backtrackingFactor \in (0, 1)$ be problem-independent constants.
	Suppose that $\alpha_k \in (0,1]$ provides sufficient objective reduction, i.e.,
	\begin{subequations}
	\begin{flalign}\label{eq:sufficient-decrease}
		\barrier(x_k + \alpha_k \dir{x}_k) \leq \barrier(x_k) + \SufficientConstant \cdot \Mk(\alpha_k \dir{x}_k)
	\end{flalign}
	and either $\alpha_k = 1$ or the larger step size $\hat{\alpha}_k = \min\{1, \alpha_k / \backtrackingFactor \}$ has insufficient objective reduction, i.e., 
	\begin{flalign}\label{eq:sufficient-decrease-fails-for-larger-step-size}
		\barrier(x_k + \hat{\alpha}_k \dir{x}_k) > \barrier(x_k) + \SufficientConstant \cdot \Mk(\hat{\alpha}_k \dir{x}_k).
	\end{flalign}
\end{subequations}
\end{condition}

Similar to the termination criteria in \citet[Algorithm 1]{hinder2024worst}, we consider \Cref{alg:IPM} converged 
when it reaches an approximate first-order stationary interior point (SIP).
Note that Definition 2 of \citet{hinder2024worst} also includes terms $\tau_l$ and $\tau_c$ 
to trade the tolerance of the complementarity conditions and duality gap, we omit these for simplicity.

\begin{definition}\cite[Definition 2]{hinder2024worst}\label{eq:first-order-SIP-full:first-order} 
A $\mu$-approximate first-order stationary interior point (SIP) is a point $(x, y)$ defined by
\begin{flalign}
(\cons(x), y) &>  \zeros \tag{SIP1.a} \label{eq:first-order-SIP-full:feasible} \\
\abs{y_i \cons_i(x) - \mu} &\le \frac{\mu}{2} \quad \forall i \in \ConSet \tag{SIP1.b} \label{eq:first-order-SIP-full:centrality}  \\
\norm{\grad_{x} \Lag(x,y) }_{2} &\le \mu \sqrt{\norm{y}_{1} + 1}\tag{SIP1.c} \label{eq:first-order-SIP-full:grad-lag}.
\end{flalign}
\end{definition}

While our method is a log barrier method, we do use a dual variable for termination. 
We either use $y_k$ (Termination check I) or $\hat{y}_{k+1}$ (Termination check II) where
\begin{subequations}\label{eq:define-all-algorithm-variables}
\begin{flalign}
&s_k := \cons(x_k), \quad  y_k := \mu S_k^{-1} \ones, \quad  S_k := \diag(s_k), \quad  Y_k = \diag(y_k) \\
&\dir{s}_k := \grad \cons(x_k) \dir{x}_k, \quad \dir{y}_k :=-Y_k S_k^{-1} \dir{s}_k \\
&\hat{y}_{k+1} := y_k + \dir{y}_k. \label{eq:hat-y-k+1}
\end{flalign}
\end{subequations}
Note that $y_k=\mu S_k^{-1}\ones$ is the standard dual multiplier estimate
induced by the log-barrier KKT conditions, since the barrier
stationarity conditions imply $Sy=\mu\ones$
\cite[Sec.~19.1, Eqs.~(19.4)--(19.5)]{nocedal2006numerical}.
The search directions given in \Cref{eq:define-all-algorithm-variables}
satisfy the classical linear system for computing the 
search directions of primal-dual interior point methods \cite[Sec.~19.2, Eq.~(19.6)]{nocedal2006numerical}:
\begin{subequations}\label{eq:primal-dual-interior-point-search-direction}
\begin{flalign}
(\grad_{xx}^2 \Lag(x_k,y_k) + \delta_k \eye) \dir{x}_k - \grad \cons(x_k)^\top \dir{y}_k &= -\grad_x \Lag(x_k,y_k)  \\
\grad \cons(x_k) \dir{x}_k - \dir{s}_k &= \zeros \\
S_k \dir{y}_k + Y_k \dir{s}_k &= \mu \ones - S_k y_k .
\end{flalign}
\end{subequations}
These two termination checks play different roles. 
Termination check I is equivalent to checking 
if $\| \grad \barrier(x_k) \| \le \varepsilon_k$ and is needed to establish 
that $\| \grad \barrier(x_k) \| > \varepsilon_k$ so we can satisfy 
Condition~\ref{condition:search-dir}; on the other hand, 
termination check II is used in our convergence analysis
to establish that the algorithm always reduces the log barrier 
or terminates.

\begin{algorithm}
	\caption{\textbf{A}daptive
			\textbf{L}ine \textbf{S}earch \textbf{I}nterior \textbf{P}oint \textbf{M}ethod}
	\label{alg:IPM}
	\begin{algorithmic}[1]
		\Function{AdaptiveLineSearchIPM}{$x_1,\mu$}
			\State \textbf{Input requirements:} $x_1 \in \interior{\X}$, $\mu > 0$
			\State \textbf{Problem-independent parameters:} $\SufficientConstant \in (0, 1/2)$, $\backtrackingFactor \in (0, 1)$, $\eta_1 \in (0, 1/2)$, $\eta_2 \in (\eta_1 , 1/2)$
			\State \textbf{Requirements:} $2\eta_2 + \frac{5}{8} (1 - \SufficientConstant) \eta_1 \le \sqrt{7/8}$
			\For{$k = 1, \dots, \infty$}
				\State \textbf{Termination check I:}
				\Indent 
					\State Calculate $y_{k}$ as per \Cref{eq:define-all-algorithm-variables}
					\If{$(x_{k},y_{k})$ is a $\mu$-approximate stationary interior point (\Cref{eq:first-order-SIP-full:first-order})}
						\State \Return $(x_{k},y_{k})$
					\EndIf
				\EndIndent
				\State 
				\State \textbf{Step of regularized Newton with line search:}
				\Indent 
					\State Compute a regularized Newton search direction $\dir{x}_k$ that satisfies Condition~\ref{condition:search-dir}
					\State Find a step size $\alpha_k$ that satisfies Condition~\ref{step-size-condition}
					\State $x_{k+1} \gets x_k + \alpha_k \dir{x}_k$
				\EndIndent
				\State
				\State \textbf{Termination check II:}
				\Indent 
					\State Calculate $\hat{y}_{k+1}$ as per \Cref{eq:define-all-algorithm-variables}
					\If{$(x_{k+1},\hat{y}_{k+1})$ is a $\mu$-approximate stationary interior point (\Cref{eq:first-order-SIP-full:first-order})}
						\State \Return $(x_{k+1},\hat{y}_{k+1})$
					\EndIf
				\EndIndent
			\EndFor
		\EndFunction
	\end{algorithmic}
\end{algorithm}

We will use \Cref{lem:model-alpha-k-bound}, which is a well-known fact about convex functions, in our analysis.

\begin{fact}[Well-known]\label{lem:model-alpha-k-bound}
If $\barrier$ is convex and $\dir{x}_k$ satisfies~\eqref{eq:search-direction} with $\delta_k \ge 0$, then for all $\alpha \in [0,1]$,
$\Mk(\alpha \dir{x}_k) \le \alpha \Mk(\dir{x}_k) \le \frac{1}{2}\Mk( \alpha  \dir{x}_k)$.
\end{fact}

\begin{proof}
By convexity we have 
$\Mk(\alpha \dir{x}_k) = \alpha \grad \barrier(x_k)^\top \dir{x}_k + 
\frac{\alpha^2}{2} (\dir{x}_k)^\top \grad^2 \barrier(x_k) \dir{x}_k \le 
\alpha \grad \barrier(x_k)^\top \dir{x}_k + \frac{\alpha}{2} (\dir{x}_k)^\top \grad^2 \barrier(x_k) \dir{x}_k = 
\alpha \Mk(\dir{x}_k)$.
Moreover,
$\grad \barrier(x_k)^\top \dir{x}_k = - (\dir{x}_k)^\top \grad^2 \barrier(x_k) \dir{x}_k - \delta_k \|\dir{x}_k\|^2$.
Thus, $\frac{1}{2}\Mk(\alpha\dir{x}_k)-\alpha\Mk(\dir{x}_k)
=
\frac{\alpha\delta_k}{2}\|\dir{x}_k\|^2+\frac{\alpha^2}{4} (\dir{x}_k)^\top \grad^2 \barrier(x_k)\dir{x}_k
\ge 0$
where the inequality uses $\delta_k\ge 0$ and convexity of $\barrier$.
\end{proof}

\subsection{A warm-up: analysis for unconstrained convex optimization}\label{sec:analysis-our-adaptive-line-search-ipm-unconstrained}

This section analyzes \Cref{alg:IPM} in the unconstrained setting. 
We obtain slightly stronger results in this special setting than we do in 
\Cref{sec:iterations-bound-Fritz-John-points}
which analyzes the constrained setting.
This section also serves as a ``warm-up'' to exhibit the basic idea of our minimum step size 
threshold proof technique
before proceeding to the more technically involved constrained setting analyzed in \Cref{sec:iterations-bound-Fritz-John-points}. 
Dropping the constraints from~\eqref{eq:optimization-constrained-problem} gives $\barrier = f$ and 
therefore the problem is simply to find a stationary point of $f$. 
Moreover, in the absence of constraints, the norm of dual variables in the formula 
$\varepsilon_k =\mu\sqrt{1+\|y_k\|_1}$ is set to zero, i.e., 
$\varepsilon_k =\mu$.
For our analysis, we will assume that the Hessian of $f$ is $\LipHess$-Lipschitz.
A well-known fact \citep[Lemma 1]{nesterov2006cubic} that we will find useful is that, if the Hessian of $f$ is $\LipHess$-Lipschitz, then
\begin{subequations}
\begin{flalign}
\| \grad f(x_k+\dir{x}_k) \|_2 &\le \| \grad \Mk(\dir{x}_k) \|_2 + \frac{\LipHess}{2} \| \dir{x}_k \|_2^2  \label{eq:Lip-grad1-journal-IPM-unconstrained}\\
f(x_k + \alpha \dir{x}_k) &\le f(x_k) + \Mk(\alpha \dir{x}_k) + \frac{\LipHess \alpha^3}{6} \| \dir{x}_k \|_2^3 \quad \forall \alpha \in [0,1] \label{eq:Lip-progress1-journal-IPM-unconstrained} 
\end{flalign}
\end{subequations}
for some $\LipHess \in (0,\infty)$.

\paragraph{Minimum step size threshold.}
Our proof revolves around the following minimum step size threshold,
\begin{align}
\alphaInd{k} := \sqrt{\frac{\eta_1^2 \mu^{3} \min\{6 (1 - \SufficientConstant) \eta_1, 2(1-2\eta_2) \}}{\LipHess (-\mk)^2}}. \label{alpha-min-def-IPM-unconstrained}
\end{align}
Our proof strategy is as follows. 
First, in \Cref{lemma-backtracking-threshold-and-trigger-phase} we show that any step size below this threshold will trigger 
the sufficient decrease condition (\Cref{eq:sufficient-decrease}), allowing 
us to conclude that if $\alphaInd{k} \ge 1$ then $\alpha_k=1$ and if $\alphaInd{k} < 1$ then 
$\alpha_k  \ge \backtrackingFactor \cdot \alphaInd{k}$.
For iterations where $\alphaInd{k} < 1$, \Cref{lemma-function-reduction-using-model-bound-IPM-unconstrained} uses 
the sufficient decrease condition to show the function value reduction per iteration is at least 
$-\alpha_{k} \SufficientConstant \mk \ge -\backtrackingFactor \alphaInd{k} \SufficientConstant \mk = \Omega(\mu^{3/2})$.
Finally, in \Cref{lem:gradient-bound_model-IPM-unconstrained} we show if $\alphaInd{k} \ge 1$ then the algorithm 
terminates on iteration $k$.  
From here, \Cref{thm:bounded-initial-suboptimality-unconstrained-case} uses a standard argument to 
obtain an $O(\mu^{-3/2})$ bound on the number of iterations of the method.

We emphasize this proof technique differs from the standard 
analysis of adaptive second-order methods \cite{cartis2011adaptiveII,curtis2017trust,grapiglia2017regularized,cartis2019universal,hamad2022consistently}.
Standard methods control the norm of the search direction, i.e., $\| \dir{x}_k \|_2$
using either regularization or a trust region radius, and 
ensure sufficient function value reduction at accepted iterates. 
In contrast, our method does not directly control the size of the search direction and
relies on a line search to ensure 
sufficient progress at each iteration. Thus, it is more natural 
for us to lower bound the search direction size.

\begin{lemma} \label{lemma-backtracking-threshold-and-trigger-phase}
Suppose that the problem is unconstrained, i.e., $\barrier(x) = f(x)$ and $\grad^2 f$ is $\LipHess$-Lipschitz. 
Then, if $\alphaInd{k} \ge 1$ then $\alpha_k=1$ and if $\alphaInd{k} < 1$ then $\alpha_k  \ge \backtrackingFactor \cdot \alphaInd{k}$.
\end{lemma}
\begin{proof}
Suppose that $\alpha \in (0,1]$ satisfies $\alpha \leq \alphaInd{k}$. Then,
\[
\frac{\alpha^3 \LipHess \|\dir{x}_k\|_2 ^ 3}{6} \le \frac{\alpha \alphaInd{k}^2 \LipHess \|\dir{x}_k\|_2 ^ 3}{6} \underle{(a)} 
\alpha \frac{\LipHess \alphaInd{k}^2 (\eta_1 \mu)^{-3} (-\mk)^3}{6} \underle{(b)} (1- \SufficientConstant)  \alpha (-\mk)  
\underle{(c)} (1 - \SufficientConstant) (-\Mk( \alpha \dir{x}_k))
\]
where $(a)$ uses \Cref{eq:direction-condition-eta1}, $(b)$ uses \Cref{alpha-min-def-IPM-unconstrained}, and $(c)$ uses \Cref{lem:model-alpha-k-bound}. Combining this inequality with \Cref{eq:Lip-progress1-journal-IPM-unconstrained} it follows that 
\[
f(x_k + \alpha \dir{x}_k) - f(x_k) \le \Mk(\alpha \dir{x}_k)
+ \frac{\alpha ^ 3\LipHess}{6} \| \dir{x}_k \|_2^3
\le  
\SufficientConstant \cdot \Mk(\alpha \dir{x}_k),
\]
i.e., the sufficient progress condition (\Cref{eq:sufficient-decrease}) holds. 
We conclude that $\alpha \in (0,1]$ and $\alpha \leq \alphaInd{k}$ implies \Cref{eq:sufficient-decrease} holds. 

This conclusion allows us to deduce the desired result.
If $\alphaInd{k}\ge 1$, then the step size $\alpha=1$ satisfies Condition~\ref{step-size-condition}, and so $\alpha_k=1$. 
If $\alphaInd{k}<1$, then $\alpha_k \ge \backtrackingFactor \alphaInd{k}$; otherwise $\alpha = \alpha_k/\backtrackingFactor<\alphaInd{k}$ 
would also satisfy \Cref{eq:sufficient-decrease}, contradicting the second part of Condition~\ref{step-size-condition}. 
\end{proof}

\begin{lemma}\label{lemma-function-reduction-using-model-bound-IPM-unconstrained}
Suppose that the problem is unconstrained, i.e., $\barrier(x) = f(x)$ and $\grad^2 f$ is $\LipHess$-Lipschitz. 
If $\alphaInd{k} < 1$ then
	\begin{flalign*}
		f(x_k) - f(x_{k+1}) \geq \hat{C} \cdot \LipHess^{-1/2} \mu^{3/2}
	\end{flalign*}
	where $\hat{C} := \backtrackingFactor \SufficientConstant \eta_1 \cdot \min\{ \sqrt{6(1- \SufficientConstant) \eta_1} , \sqrt{2(1-2\eta_2)}\}$ is a problem-independent constant.
\end{lemma}
\begin{proof}
By $(a)$ equation~\eqref{eq:sufficient-decrease} and \Cref{lem:model-alpha-k-bound}, 
$(b)$, \Cref{lemma-backtracking-threshold-and-trigger-phase} with $\alphaInd{k} < 1$, we get
\begin{flalign*}
f(x_k) - f(x_{k+1}) &\underge{(a)}  -\SufficientConstant \cdot \alpha_{k} \mk 
\underge{(b)} -\backtrackingFactor \SufficientConstant \alphaInd{k} \mk  
= \backtrackingFactor \SufficientConstant  \sqrt{\frac{\eta_1^2 \mu^{3} \min\{6 (1 - \SufficientConstant) \eta_1, 2(1-2\eta_2) \}}{\LipHess (-\mk)^2}} \cdot (-\mk)  \\
&= \hat{C} \LipHess^{-1/2} \mu^{3/2}.
\end{flalign*}
\end{proof}

\begin{lemma} \label{lem:gradient-bound_model-IPM-unconstrained}
Suppose that the problem is unconstrained, i.e., $\barrier(x) = f(x)$ and $\grad^2 f$ is $\LipHess$-Lipschitz. 
If $\alphaInd{k} \ge 1$ then $\| \grad f(x_{k+1}) \|_2  \le \mu$.
\end{lemma}
\begin{proof}
Since $\alphaInd{k} \ge 1$, \Cref{lemma-backtracking-threshold-and-trigger-phase} implies $\alpha_k=1$.
Moreover,
\begin{flalign*}
\| \grad f(x_{k+1}) \|_2  &= \| \grad f(x_k + \dir{x}_k) \|_2 \underle{(a)} \| \grad \Mk(\dir{x}_k) \|_2 + \frac{\LipHess}{2} \| \dir{x}_k \|_2^2 \undereq{(b)} \delta_k \|\dir{x}_k\|_2 + \frac{\LipHess}{2} \| \dir{x}_k \|_2^2 \\
&\underle{(c)} 2\eta_2 \mu + \frac{\LipHess}{2} \| \dir{x}_k \|_2^2 
\underle{(d)} 2\eta_2 \mu + \frac{\LipHess}{2} (\eta_1 \mu)^{-2} \cdot (-\mk)^2 \le 2\eta_2 \mu + \frac{\LipHess}{2} (\eta_1 \mu)^{-2} \cdot (-\mk)^2 \alphaInd{k}^2 \\
&\le 2 \eta_2 \mu + \frac{\LipHess}{2} (\eta_1 \mu)^{-2} \cdot (-\mk)^2 \cdot \frac{\eta_1^2 \mu^{3}  2(1-2\eta_2)}{\LipHess (-\mk)^2}
 \le \mu
\end{flalign*}
where inequality $(a)$ uses \Cref{eq:Lip-grad1-journal-IPM-unconstrained}, 
equality $(b)$ uses $\| \grad \Mk(\dir{x}_k) \|_2 = \delta_k \|\dir{x}_k\|_2$, 
inequality $(c)$ uses \Cref{eq:direction-condition-eta2} from Condition~\ref{condition:search-dir} and the fact that $\alpha_k = 1$, 
and inequality $(d)$ uses the upper bound on $\| \dir{x}_k \|$ given in \Cref{eq:direction-condition-eta1}. 	
\end{proof}

\begin{theorem}\label{thm:bounded-initial-suboptimality-unconstrained-case}
Suppose that the problem is unconstrained, i.e., $\barrier(x) = f(x)$ and $\grad^2 f$ is $\LipHess$-Lipschitz.
Let $\hat{C} := \backtrackingFactor \SufficientConstant \eta_1 \cdot \min\{ \sqrt{6(1- \SufficientConstant) \eta_1} , \sqrt{2(1-2\eta_2)}\}$ be a problem-independent constant. 
Then, \Cref{alg:IPM} terminates after at most 
	\begin{flalign*}
		K \le 1+\hat{C} ^ {-1} \cdot \frac{(f(x_1) - f_\star) \LipHess ^ {1/2}}{\mu ^ {3/2}}
	\end{flalign*}
iterations.
\end{theorem}
\begin{proof}
Consider some iteration $K$ of the algorithm, since the algorithm 
has not terminated across iterations $k = 1, \dots, K-1$,
by the contrapositive of \Cref{lem:gradient-bound_model-IPM-unconstrained}
we deduce that $\alphaInd{k} < 1$ for $k \in [1,K-1]$.
Thus, by \Cref{lemma-function-reduction-using-model-bound-IPM-unconstrained}  and a telescoping sum, we get:
\begin{flalign*}
f(x_1) - f(x_{K}) \ge \sum_{k=1}^{K-1} (f(x_k) - f(x_{k+1})) \geq \sum_{k=1}^{K-1} \hat{C} \cdot \LipHess^{-1/2} \mu^{3/2} = (K-1) \cdot \hat{C} \cdot \LipHess^{-1/2} \mu^{3/2}.
\end{flalign*}
Rearranging the above inequality and using $f(x_{K}) \ge f_\star$ gives the desired result.
\end{proof}

\begin{remark}
In convex optimization, it is well known that if the initial sublevel set is bounded by $D$, i.e., 
$D = \diam(\{ x \in \R^{\NumVar} : f(x) \le f(x_1) \})$ then the iteration bound for cubic regularized Newton to find an $\epsilon$-approximate minimizer is 
$O(\LipHess^{1/2} D^{3/2} \epsilon^{-1/2})$ \cite{nesterov2006cubic}. This may be preferable to \Cref{thm:bounded-initial-suboptimality-unconstrained-case}
which obtains a worse $\mu^{-3/2}$ scaling\footnote{Although the assumptions are slightly different, in particular, \Cref{thm:bounded-initial-suboptimality-unconstrained-case} assumes only that the initial suboptimality is bounded.}. 
However, in \Cref{thm:unconstrained-case} we show we can 
match this bound of cubic regularized Newton by annealing $\mu$.
Additionally, it is well-known that many methods \cite{nesterov2006cubic,cartis2011adaptiveII,curtis2017trust} can achieve a similar result to \Cref{thm:bounded-initial-suboptimality-unconstrained-case} without 
a convexity assumption.
\end{remark}

\subsection{Lemmas on local approximations and direction sizes}\label{sec:lemmas-local-approximation-direction-size}

This section lists a series of useful lemmas from \citet{hinder2024worst} for analyzing log barrier methods. 
They depend on the following assumption.
 
\begin{assumption}[Lipschitz derivatives]\label{assume-lip-deriv}
	The functions $f$ and each $\cons_i$ for $i \in [m]$ are thrice differentiable on $\reals^{\NumVar}$.
	Let $\LipGrad, \LipHess \in (0,\infty)$.
	On the set $\X$, $f$ and each  $\cons_i$ have $\LipGrad$-Lipschitz first derivatives and $\LipHess$-Lipschitz second derivatives.
\end{assumption}

\Cref{lemPrimalDualApproxResultTwo} bounds the gap between 
the predicted reduction from the second order model $\Mk$ and the actual reduction of the log barrier $\barrier$. This allows us to establish a lower bound on the actual reduction in the log barrier similar to \Cref{lemma-function-reduction-using-model-bound-IPM-unconstrained} in the unconstrained setup.
Recall from \Cref{eq:define-all-algorithm-variables} that 
\begin{flalign*}
&s_k := \cons(x_k), \quad  y_k := \mu S_k^{-1} \ones, \quad  S_k := \diag(s_k), \quad  Y_k = \diag(y_k), \\
&\dir{s}_k := \grad \cons(x_k) \dir{x}_k, \quad \dir{y}_k :=-Y_k S_k^{-1} \dir{s}_k
\end{flalign*}
and that $S_k y_k + S_k \dir{y}_k + Y_k \dir{s}_k = \mu \ones$. Also, by definition of \Cref{alg:IPM} we have 
$x_k \in \interior{\X}$.

\begin{lemma}\cite[Lemma 5]{hinder2024worst}\label{lemPrimalDualApproxResultTwo}
	Suppose Assumption~\ref{assume-lip-deriv} holds (Lipschitz derivatives). Let $\alpha \in [0,1]$, and 
	\begin{flalign}\label{eq:kappa-requirement}
	\kappa_{\alpha} := \| S_k^{-1} \dir{s}_k \|_{2} \alpha +  \frac{\LipGrad \| \dir{x}_k \|_2^2  \| y_k \|_2 \alpha^2}{\mu}.
	\end{flalign}
	If $\kappa_{\alpha} \in (0, 1/4]$ then {$\Convex\{ x_k, x_k + \alpha \dir{x}_k \} \subseteq \X$} and
	\begin{flalign}\label{eq:progress-barrier-error}
| \barrier(x_k) + \Mk(\alpha \dir{x}_k) -  \barrier(x_k + \alpha \dir{x}_k) | \le  \frac{\LipHess ( 1 +  2 \| y_k \|_1 ) \| \dir{x}_k \|_2^3 \alpha^3}{6}  + {\frac{4 \LipGrad \| \dir{x}_k \|_2^2 \| y_k \|_1 \kappa_{\alpha} \alpha^2 }{3}} + \frac{5 \mu \kappa_{\alpha}^3}{3}.
	\end{flalign}
\end{lemma}

\Cref{lem:complementary-and-gradient-Lagrangian-error} allows us to establish when we have found an 
approximate stationary interior point, i.e., \Cref{eq:first-order-SIP-full:first-order} holds.

\begin{lemma}\cite[Lemma 6 \& 7]{hinder2024worst}\label{lem:complementary-and-gradient-Lagrangian-error}
	Suppose Assumption~\ref{assume-lip-deriv} holds (Lipschitz derivatives). Let $\hat{x}_{k+1} = x_{k} + \dir{x}_k$, $\hat{y}_{k+1} = y_k + \dir{y}_k$, $\Convex\{x_k, \hat{x}_{k+1} \} \subseteq \X$, $\hat{S}_{k+1} = \diag(\cons(\hat{x}_{k+1}))$, then
	\begin{flalign}
		\| \hat{S}_{k+1} \hat{y}_{k+1} - \mu \ones \|_{2} &\le  \mu \| S_k^{-1} \dir{s}_k \|_{2}^2 + \frac{\LipGrad}{2}  \| y_k \|_{2} (1 + \| S_k^{-1} \dir{s}_k \|_{2}) \| \dir{x}_k \|_2^2.
		\label{eq:new-general-comp-bound}
	\end{flalign}
Moreover, 
\begin{equation}\label{eq:LagError}
\begin{split}
&\| \grad_{x} \Lag( x_k , y_k ) + \grad_{xx}^2 \Lag (x_k, y_k)^\top  \dir{x}_k - (\grad \cons(x_k))^\top \dir{y}_k   - \grad_{x} \Lag(\hat{x}_{k+1}, \hat{y}_{k+1})
\|_{2} \\
&\le \LipGrad \| y_k \|_2  \| \dir{x}_k \|_2 \| S_k^{-1} \dir{s}_k \|_2 + \frac{\LipHess}{2}  (\| y_k \|_1 + 1) \|  \dir{x}_k \|^2_2.
\end{split}
\end{equation}
\end{lemma}

Finally, we will finish this section with \Cref{lem:invS-ds-bound}. This is a key result that allows us to bound the direction size of the slack variables, i.e, $\|S_k^{-1} \dir{s}_k\|_2$. Combining \Cref{lem:invS-ds-bound} with the above lemmas will allow us to give concrete bounds on the reduction of the log barrier at each iteration.
\begin{lemma}\cite[Lemma 9]{hinder2024worst}\label{lem:invS-ds-bound} 
	Suppose Assumption~\ref{assume-lip-deriv} holds (Lipschitz derivatives). Then the following inequality holds:\begin{flalign}\label{eq:bound-slack-variable}
		\| S_k^{-1} \dir{s}_k \|_2 \leq \sqrt{\frac{-2\mk}{\mu}}.
	\end{flalign}
\end{lemma}

\begin{proof}
As discussed immediately after Lemma 9 of \citet{hinder2024worst} one has 
$\| S^{-1}_k \dir{s}_k \|_2 \le \sqrt{ \frac{ -(\dir{x}_k)^\top \grad_{xx}^2 \Lag(x_k,y_k) \dir{x}_k - 2 \mk}{\mu} }$.
Since $\grad_{xx}^2 \Lag(x_k,y_k) \succeq 0$ by convexity we get the desired result.
\end{proof}

\subsection{Analysis of our log barrier method with fixed barrier parameter}\label{sec:iterations-bound-Fritz-John-points}

In this subsection, we prove that \Cref{alg:IPM} finds an approximate stationary interior point in at most $O(\mu^{-5/3})$ iterations. 
This proof shares a similar structure to the proof in \Cref{sec:analysis-our-adaptive-line-search-ipm-unconstrained}
for the unconstrained setup but is more technical. 
In particular, we show that there is a threshold representing the minimum step size that the algorithm must take as a function of the expected progress $-\mk$.
Although this step size threshold decays as $-\mk$ grows, this is balanced out because,
by the sufficient decrease condition (\Cref{eq:sufficient-decrease}) and \Cref{lem:model-alpha-k-bound}, the function value is reduced by at least $\alpha_k \SufficientConstant \mk$
at each iteration.
On the other hand, we show if this step threshold exceeds one then $\alpha_k=1$ and 
the algorithm will terminate at the next iteration.
While the high-level structure is similar to \Cref{sec:analysis-our-adaptive-line-search-ipm-unconstrained},
the threshold and exact analysis is much more involved due to the presense of the log barrier.

\paragraph{Step size thresholds that guarantee feasibility.} 
\begin{align}\label{def:alpha-feas-threshold}
	\alphaInd{k}^{(1)} := \frac{1}{8} \sqrt{\frac{\mu}{-2\mk}}, \qquad
	\alphaInd{k}^{(2)} := \frac{1}{2\sqrt{2 \LipGrad}} \cdot \frac{\mu ^ {3/2} \cdot \eta_1}{-\mk}.
\end{align}
In \Cref{lemPrimalDualApproxResultTwoAlpha} we show that for any step size $\alpha \in (0,1]$ below these thresholds, 
$\kappa_{\alpha} \in (0,1/4]$ and therefore, by \Cref{lemPrimalDualApproxResultTwo} the candidate 
iterate $x_{k} + \alpha \dir{x}_k$ is feasible.

\begin{lemma}\label{lemPrimalDualApproxResultTwoAlpha}
	Suppose Assumption~\ref{assume-lip-deriv} holds (Lipschitz derivatives). Then, for any step size $\alpha \in [0,1]$ with $\alpha \le \min\{ 1, \alphaInd{k}^{(1)}, \alphaInd{k}^{(2)} \}$, we have
	$\kappa_{\alpha} \in [0,1/4]$ where $\kappa_{\alpha}$ is defined in \Cref{lemPrimalDualApproxResultTwo}.
\end{lemma}

\begin{proof}
Using equation \eqref{eq:kappa-requirement} from \Cref{lemPrimalDualApproxResultTwo}, we will show that if the step size $\alpha$ satisfies the inequality $\alpha \leq \min \{1, \alphaInd{k}^{(1)}, \alphaInd{k}^{(2)} \}$ then $\kappa_{\alpha} \in [0,1/4]$.

Denote $S_k = \diag(\cons(x_k))$, $\dir{s}_k = \grad \cons(x_k) \dir{x}_k$, and $y_k = \mu S_k^{-1} \ones$.
First, we will show that $\alpha \| S_k^{-1} \dir{s}_k \|_{2} \leq \frac{1}{8}$. Note that:
\begin{flalign*}
	\alpha \| S_k^{-1} \dir{s}_k \|_{2} &\underle{(a)} \frac{1}{8} \cdot  \sqrt{\frac{\mu}{-2\mk}} \cdot \| S_k^{-1} \dir{s}_k \|_{2} \\
	&\underle{(b)} \frac{1}{8} \cdot \sqrt{\frac{\mu}{-2\mk}} \cdot \sqrt{\frac{-2\mk}{\mu}} = \frac{1}{8}
\end{flalign*}
where inequality $(a)$ uses $\alpha \leq \min \{\alphaInd{k}^{(1)}, \alphaInd{k}^{(2)}\} \leq \alphaInd{k}^{(1)}$ and the definition of  $\alphaInd{k}^{(1)}$ from \eqref{def:alpha-feas-threshold} and inequality $(b)$ uses $\| S_k^{-1} \dir{s}_k \|_2 \leq \sqrt{\frac{-2\mk}{\mu}}$ from \eqref{eq:bound-slack-variable}.
	
Next, we will show that $\frac{\LipGrad \| \dir{x}_k \|_2^2  \| y_k \|_2 \alpha ^ 2}{\mu} \le \frac{1}{8}$. Note that:
\begin{flalign*}
\frac{\LipGrad \|\dir{x}_k \|_2^2  \| y_k \|_2 \alpha^2}{\mu} &\underle{(a)} \frac{\LipGrad \|\dir{x}_k \|_2^2  \| y_k \|_2}{\mu} \cdot \left(\frac{1}{2\sqrt{2 \LipGrad}} \cdot \frac{\mu ^ {3/2} \cdot \eta_1}{-\mk}\right) ^ 2 \\
&= \frac{\| y_k \|_2}{8} \cdot \|\dir{x}_k \|_2^2 \cdot \frac{\mu ^ 2 \eta_1 ^ 2}{(-\mk)^2} \\
&\underle{(b)} \frac{1}{8} \cdot \|\dir{x}_k \|_2^2 \cdot \mu ^ 2 \eta_1 ^ 2 \cdot \frac{\|y_k\|_2}{\eta_1 ^ 2 \mu ^ 2 \|\dir{x}_k\|_2^2 (\|y_k\|_1 + 1)} =  \frac{1}{8} \cdot \frac{\|y_k\|_2}{\|y_k\|_1 + 1} \le \frac{1}{8}
\end{flalign*}
where inequality $(a)$ uses $\alpha \leq \min \{\alphaInd{k}^{(1)}, \alphaInd{k}^{(2)}\} \leq \alphaInd{k}^{(2)}$ and the definition of $\alphaInd{k}^{(2)}$ from \eqref{def:alpha-feas-threshold} and inequality $(b)$ uses~\eqref{eq:direction-condition-eta1} from Condition~\ref{condition:search-dir} with $\varepsilon_k = \mu \sqrt{\|y_k\|_1 + 1}$.  
	
Therefore $\kappa_{\alpha} = \alpha \| S_k^{-1} \dir{s}_k \|_{2} + \frac{\LipGrad \| \dir{x}_k \|_2^2  \| y_k \|_2 \alpha ^ 2}{\mu} \leq \frac{1}{4} $ as desired.
\end{proof}

\paragraph{Step size thresholds that guarantee sufficient decrease.}  
\begin{align}\label{def:alpha-sufficient-decrease-threshold}
\begin{split}
\alphaInd{k}^{(3)} := \frac{(1 - \SufficientConstant)^{1/2} \mu^{3/2} \eta_1^{3/2}}{\LipHess^{1/2} (-\mk)},
\qquad
\alphaInd{k}^{(4)} := \frac{(1 - \SufficientConstant)^{1/2} \mu^{5/4} \eta_1}{(8\sqrt{2})^{1/2} \LipGrad^{1/2} (-\mk)^{3/4}}, \qquad\quad\qquad\quad \\[2ex]
\alphaInd{k}^{(5)} := \frac{(1 - \SufficientConstant)^{1/3} \mu^{5/3} \eta_1^{4/3}}{2 \LipGrad^{2/3} (-\mk)},
\quad 
\alphaInd{k}^{(6)} := \frac{(1 - \SufficientConstant)^{1/2} \mu^{1/4}}{(80\sqrt{2})^{1/2} (-\mk)^{1/4}}, 
\quad \alphaInd{k}^{(7)} := \frac{(1 - \SufficientConstant)^{1/5} \mu^{8/5} \eta_1^{6/5}}{40^{1/5} \LipGrad^{3/5} (-\mk)}.
\end{split}
\end{align}
\Cref{lemMinimumStepSizeIPM} uses these thresholds and the thresholds defined in \Cref{def:alpha-feas-threshold} to 
lower bound the step size $\alpha_k$ taken by our method. The proof of \Cref{lemMinimumStepSizeIPM} shows that for 
any step size $\alpha \in (0,1]$ below these thresholds and the thresholds defined in \Cref{def:alpha-feas-threshold}, the sufficient decrease condition given in the first part of Condition~\ref{step-size-condition} is satisfied. 
Combining this fact with the second part of Condition~\ref{step-size-condition} which requires that larger step sizes   
violate the sufficient decrease condition, allows us to obtain our desired lower bounds on the step size.

The analog of \Cref{lemMinimumStepSizeIPM} in the unconstrained setting is \Cref{lemma-backtracking-threshold-and-trigger-phase}.

\begin{lemma}\label{lemMinimumStepSizeIPM}
Suppose Assumption~\ref{assume-lip-deriv} holds (Lipschitz derivatives).
Then, if $\min_{i \in [7]} \alphaInd{k}^{(i)} \ge 1$ then $\alpha_k=1$ and if $\min_{i \in [7]} \alphaInd{k}^{(i)} < 1$ 
then $\alpha_k > \backtrackingFactor \cdot \min_{i \in [7]} \alphaInd{k}^{(i)}$.
\end{lemma}

To prove \Cref{lemMinimumStepSizeIPM} we first prove Claim~\ref{claim:generic-min-step-size}.
We then recast the Taylor series error bounds from \Cref{lemPrimalDualApproxResultTwo} in the form of \Cref{eq:structured-error-bound} using \Cref{lemPrimalDualApproxResultTwoAlpha} and \Cref{eq:progress-barrier-error}, allowing us to mechanically derive \Cref{lemMinimumStepSizeIPM} using Claim~\ref{claim:generic-min-step-size}. 

\begin{claim}\label{claim:generic-min-step-size}
Let the premise of \Cref{lemMinimumStepSizeIPM} hold.
Additionally, suppose that there exists constants $t_k, \zeta_{k,1}, \dots, \zeta_{k,P} \in (0,\infty)$ and $\theta_{1}, \dots, \theta_{P} \in (1,\infty)$ such that for all $\alpha \in [0, t_k]$ 
\begin{flalign}\label{eq:structured-error-bound}
\barrier(x_k + \alpha \dir{x}_k) \le \barrier(x_k) + \Mk(\alpha \dir{x}_k) + \max_{p \in [P]} \zeta_{k,p} \alpha^{\theta_{p}}.
\end{flalign}
Let
\[
\tau_k := \min\left\{  t_k , \min_{p \in [P]} \left( \frac{-\mk (1 - \SufficientConstant)}{\zeta_{k,p}} \right)^{\frac{1}{\theta_{p}-1}} \right\}.
\]
Then, if $\tau_k \ge 1$ then $\alpha_k=1$ and if $\tau_k < 1$ then $\alpha_k > \backtrackingFactor \tau_k$.
\end{claim}
\begin{proof}
By the Claim statement and \Cref{lem:model-alpha-k-bound}, for the sufficient decrease condition, \Cref{eq:sufficient-decrease}, to hold, it suffices for 
\[
\max_{p \in [P]} \zeta_{k,p} \alpha^{\theta_{p}} \le \alpha (\SufficientConstant - 1) \cdot \mk.
\]
Dividing both sides by $\alpha$ gives:
\[
\max_{p \in [P]} \zeta_{k,p} \alpha^{\theta_{p} - 1} \le (\SufficientConstant - 1) \cdot \mk.
\]
Since $\theta_{p} > 1$, solving for $\alpha$ gives that the sufficient decrease condition (\Cref{eq:sufficient-decrease}) is satisfied for all
\[
\alpha \le \min\left\{ t_k, \min_{p \in [P]} \left( \frac{- \mk (1 - \SufficientConstant)}{\zeta_{k,p}} \right)^{\frac{1}{\theta_{p} - 1}} \right\} = \tau_k.
\]
Therefore if $\tau_k \ge 1$, then \Cref{eq:sufficient-decrease-fails-for-larger-step-size} fails for all $\hat{\alpha}_k \le 1$, so Condition~\ref{step-size-condition} is satisfied only for $\alpha_k = 1$. Otherwise, if $\tau_k < 1$, then by the requirement
in \Cref{eq:sufficient-decrease-fails-for-larger-step-size} that the larger step size 
$\hat{\alpha}_k = \min\{ \alpha_k / \backtrackingFactor, 1 \}$ fails to obtain sufficient decrease we deduce that 
$\alpha_k / \backtrackingFactor > \tau_k$ and thus $\alpha_k > \backtrackingFactor \tau_k$.
\end{proof}

\begin{proof}[Proof of \Cref{lemMinimumStepSizeIPM}]
We proceed to obtain a bound in the form of \Cref{eq:structured-error-bound} so we can 
employ Claim~\ref{claim:generic-min-step-size}.
For simplicity, let $\ykPlusOne = \sqrt{\| y_k \|_1 + 1}$ where $y_k = \mu S_k^{-1} \mathbf{1}$ and $S_k = \diag(\cons(x_k))$. Using \eqref{eq:direction-condition-eta1} from 
Condition~\ref{condition:search-dir} with $\varepsilon_k = \mu \ykPlusOne$, we get:
\[
\|\dir{x}_k\|_2 \le \frac{-\mk}{\varepsilon_k \eta_1} = \frac{-\mk}{\ykPlusOne \mu \eta_1}.
\]
Also, using \eqref{eq:kappa-requirement} from  \Cref{lemPrimalDualApproxResultTwo} as well as \Cref{lem:invS-ds-bound}, for any $\alpha \in [0,1]$ we get:
\[
\kappa_{\alpha} \le 2 \max\left\{ \sqrt{\frac{-2 \mk}{\mu}} \alpha,  \frac{\LipGrad \| \dir{x}_k \|_2^2  \ykPlusOne^2}{\mu}  \alpha^2 \right\} \le 2 \max\left\{ \sqrt{\frac{-2 \mk}{\mu}} \alpha,  \frac{\LipGrad (-\mk)^2}{\mu^3 \eta_1^2} \alpha^2 \right\}.
\]
Moreover, using \Cref{lem:model-alpha-k-bound} and \eqref{eq:progress-barrier-error} from \Cref{lemPrimalDualApproxResultTwo}, for any $\alpha \in [0,\min\{1,\alphaInd{k}^{(1)},\alphaInd{k}^{(2)}\}]$ we have 
\begin{flalign*}
&\barrier(x_k + \alpha \dir{x}_k) - \barrier(x_k) - \Mk(\alpha \dir{x}_k) \le 3 \max\left\{ \frac{2 \LipHess \ykPlusOne^2 \| \dir{x}_k \|_2^3 \alpha^3}{6} , {\frac{4 \LipGrad \| \dir{x}_k \|_2^2 \ykPlusOne^2 \kappa_{\alpha} \alpha^2 }{3}} , \frac{5 \mu \kappa_{\alpha}^3}{3} \right\} \\
&\underle{(a)} \max\left\{ \frac{\LipHess \alpha^3 (-\mk)^3}{\ykPlusOne \mu^3 \eta_1^3} , \frac{4 \LipGrad \kappa_{\alpha} \alpha^2 (-\mk)^2}{\mu^2 \eta_1^2} , 5 \mu \kappa_{\alpha}^3 \right\} \\
&\underle{(b)} \max\left\{ 
\frac{\LipHess (-\mk)^3}{\ykPlusOne \mu^3 \eta_1^3} \alpha^3, 
\frac{8\sqrt{2} \LipGrad  (-\mk)^{5/2}}{\mu^{5/2} \eta_1^2} \alpha^3, 
\frac{8 \LipGrad^2 (-\mk)^4}{\mu^5 \eta_1^4}  \alpha^4, 
\frac{80\sqrt{2} (-\mk)^{3/2}}{\mu^{1/2}} \alpha^3 , 
\frac{40 \LipGrad^3 (-\mk)^6}{\mu^8 \eta_1^6}  \alpha^6 
\right\}.
\end{flalign*}
where inequality $(a)$ uses our upper bound on $\| \dir{x}_k \|_2$ and $(b)$ substitutes for our upper bound on $\kappa_{\alpha}$.
Thus, by \Cref{lemPrimalDualApproxResultTwoAlpha} and Claim~\ref{claim:generic-min-step-size} with $t_k = \min\{ 1, \alphaInd{k}^{(1)}, \alphaInd{k}^{(2)} \}$ we get that
\begin{align*}
\tau_k = \min\Bigg\{
&t_k,
\frac{(1 - \SufficientConstant)^{1/2} \ykPlusOne^{1/2} \mu^{3/2} \eta_1^{3/2}}{\LipHess^{1/2} (-\mk)},
\frac{(1 - \SufficientConstant)^{1/2} \mu^{5/4} \eta_1}{(8\sqrt{2})^{1/2} \LipGrad^{1/2} (-\mk)^{3/4}},
\frac{(1 - \SufficientConstant)^{1/3} \mu^{5/3} \eta_1^{4/3}}{2 \LipGrad^{2/3} (-\mk)}, \\
&\frac{(1 - \SufficientConstant)^{1/2} \mu^{1/4}}{(80\sqrt{2})^{1/2} (-\mk)^{1/4}},
\frac{(1 - \SufficientConstant)^{1/5} \mu^{8/5} \eta_1^{6/5}}{40^{1/5} \LipGrad^{3/5} (-\mk)}
\Bigg\} \ge \min\left\{1, \min_{i \in [7]} \alphaInd{k}^{(i)} \right\}.
\end{align*}
\end{proof}

\newcommand{\remainder}{r}

\newcommand{\MainConst}{\frac{8 \LipGrad^{2/3}}{(1 - \SufficientConstant)^{2/3} \eta_1^{4/3}}}

\Cref{lemMinimumStepSizeIPM} showed that the step size $\alpha_k$ taken by the method 
is one if $\min_{i \in [7]} \alphaInd{k}^{(i)} \ge 1$ and otherwise 
greater than $\backtrackingFactor \min_{i \in [7]} \alphaInd{k}^{(i)}$. 
\Cref{lemBarrierProgress} shows in the case that $\min_{i \in [7]} \alphaInd{k}^{(i)} < 1$
that there is significant reduction in the barrier function. 
The analog of  
\Cref{lemBarrierProgress} in the unconstrained setting is 
\Cref{lemma-function-reduction-using-model-bound-IPM-unconstrained}.

\begin{lemma}\label{lemBarrierProgress}
Suppose Assumption~\ref{assume-lip-deriv} holds. 
If $\min_{i \in [7]} \alphaInd{k}^{(i)} < 1$ then
\begin{flalign*}
\barrier(x_k) - \barrier(x_{k+1}) \ge \SufficientConstant \backtrackingFactor \mu \left( \MainConst{} \mu^{-2/3} + \remainder(\mu) \right)^{-1}
\end{flalign*}
where
\begin{align}
\begin{split}	
\label{eq:formula-for-remainder}
\remainder(\mu)
:=
128
&+
\frac{12800}{(1-\SufficientConstant)^2}
+
\left(\frac{2\sqrt{2 \LipGrad}}{\eta_1}
+
\frac{\LipHess^{1/2}}{(1-\SufficientConstant)^{1/2}\eta_1^{3/2}} \right) \mu^{-1/2}
+
\frac{40^{1/5}\LipGrad^{3/5}}
{(1-\SufficientConstant)^{1/5}\eta_1^{6/5}} \mu^{-3/5}.
\end{split}
\end{align}
\end{lemma}
To prove \Cref{lemBarrierProgress} we first prove Claim~\ref{claim:min-barrier-progress} which
recasts the $\alphaInd{k}$ as a generic formula, allowing us to produce a generic bound on 
$\barrier(x_k) - \barrier(x_{k+1})$.
We can then use Claim~\ref{claim:min-barrier-progress} to mechanically 
derive \Cref{lemBarrierProgress}.

\begin{claim}\label{claim:min-barrier-progress}
Suppose that there exists constants $\zeta_{k,1}, \dots, \zeta_{k,P} \in (0,\infty)$ and $\theta_{1}, \dots, \theta_{P} \in (0,1]$ such that $\min_{p \in [P]}  \ \frac{\zeta_{k,p}}{(-\mk)^{\theta_{p}}} < 1$ then, if $\alpha_k \ge \backtrackingFactor \min_{p \in [P]} \frac{\zeta_{k,p}}{(-\mk)^{\theta_{p}}}$, we have:
\[
\barrier(x_k) - \barrier(x_{k+1}) \ge \backtrackingFactor  \SufficientConstant \cdot  \min_{p \in [P]} \zeta_{k,p} ^ \frac{1}{\theta_{p}}.
\]
\end{claim}
\begin{proof}
Let $p^* = \arg\min_{p \in [P]} \frac{\zeta_{k,p}}{(-\mk)^{\theta_{p}}}$. Since $\frac{\zeta_{k,p^*}}{(-\mk)^{\theta_{p^*}}} < 1$, we have:
\[
-\mk > \zeta_{k,p^*}^{1/\theta_{p^*}}.
\]
Using \Cref{lem:model-alpha-k-bound} and the sufficient decrease condition (\cref{eq:sufficient-decrease}):
\begin{flalign*}
\barrier(x_k) - \barrier(x_{k+1}) = \barrier(x_k) - \barrier(x_k + \alpha_k \dir{x}_k) &\geq  -\SufficientConstant \cdot \Mk(\alpha_k \dir{x}_k)  
\ge \SufficientConstant \cdot \alpha_k (-\mk) \\
&\ge \backtrackingFactor \SufficientConstant \cdot \frac{\zeta_{k,p^*}}{(-\mk)^{\theta_{p^*}}} \cdot (-\mk) \\
&= \backtrackingFactor \SufficientConstant \cdot \zeta_{k,p^*} \cdot (-\mk)^{1-\theta_{p^*}}.
\end{flalign*}
Since $0 < \theta_{p^*} \le 1$, we have $1 - \theta_{p^*} \ge 0$. Rearranging $\frac{\zeta_{k,p}}{(-\mk)^{\theta_{p}}} < 1$ gives
$-\mk > \zeta_{k,p^*}^{1/\theta_{p^*}}$ and thus
\[
\frac{\barrier(x_k) - \barrier(x_{k+1})}{\backtrackingFactor \SufficientConstant} \ge \zeta_{k,p^*} \cdot (-\mk)^{1-\theta_{p^*}} 
\ge \zeta_{k,p^*} \cdot  \left(\zeta_{k,p^*}^{1/\theta_{p^*}}\right)^{1-\theta_{p^*}}
\ge  \zeta_{k,p^*} \cdot \zeta_{k,p^*}^{(1-\theta_{p^*})/\theta_{p^*}} 
= \zeta_{k,p^*}^{1/\theta_{p^*}}
\]
as desired.
\end{proof}

\begin{proof}[Proof of \Cref{lemBarrierProgress}]
Using Claim~\ref{claim:min-barrier-progress}, we will derive the reduction in the log barrier function for all the lower bound on the step size $\alpha_k$ from Lemma~\ref{lemMinimumStepSizeIPM}, i.e, $\alpha_k \ge \backtrackingFactor \min_{i \in [7]} \alphaInd{k}^{(i)}$.
Note that by our definition of $\alphaInd{k}^{(i)} \text{ } \forall i \in [7]$, we can write each $\alphaInd{k}^{(i)}$ in the format $\frac{\zeta_{k,i}}{(-\mk)^{\theta_{i}}}$ with  $\theta_{i} \le 1$. Therefore, by Claim~\ref{claim:min-barrier-progress}, we get:
\begin{flalign*}
\barrier(x_k) - \barrier(x_{k+1}) \ge 
\SufficientConstant \backtrackingFactor \cdot \min\Bigg\{ 
&\frac{\mu}{128}, 
\frac{\mu^{3/2} \eta_1}{2\sqrt{2 \LipGrad}}, 
\frac{(1 - \SufficientConstant)^{1/2} \mu^{3/2} \eta_1^{3/2}}{\LipHess^{1/2}}, 
\frac{(1 - \SufficientConstant)^{2/3} \mu^{5/3} \eta_1^{4/3}}{(8\sqrt{2})^{2/3} \LipGrad^{2/3}}, \\
&\frac{(1 - \SufficientConstant)^{1/3} \mu^{5/3} \eta_1^{4/3}}{2 \LipGrad^{2/3}}, 
\frac{(1 - \SufficientConstant)^{2} \mu}{12800}, 
\frac{(1 - \SufficientConstant)^{1/5} \mu^{8/5} \eta_1^{6/5}}{40^{1/5} \LipGrad^{3/5}} 
\Bigg\}.
\end{flalign*}
\end{proof}

Recall that \Cref{lemBarrierProgress} shows that if   
$\min_{i \in [7]} \alphaInd{k}^{(i)} < 1$
then there is significant reduction in the barrier function.
\Cref{lemBarrierTermination} complements \Cref{lemBarrierProgress}
by showing if   
$\min_{i \in [7]} \alphaInd{k}^{(i)} \ge 1$
then the algorithm terminates.
The analog of
\Cref{lemBarrierTermination} in the unconstrained setting 
is \Cref{lem:gradient-bound_model-IPM-unconstrained}.

\begin{lemma}\label{lemBarrierTermination}
Suppose Assumption~\ref{assume-lip-deriv} holds.
If $\min_{i \in [7]} \alphaInd{k}^{(i)} \ge 1$ then, $(x_{k+1}, \hat{y}_{k+1})$ is a 
$\mu$-approximate first-order SIP where $\hat{y}_{k+1}$ is defined in \eqref{eq:hat-y-k+1}.
\end{lemma}
\begin{proof}
Recall
$\dir{s}_k := \grad \cons(x_k) \dir{x}_k$.
By our assumption that $\min_{i \in [7]} \alphaInd{k}^{(i)} \ge 1$, \Cref{lemMinimumStepSizeIPM} implies that $\alpha_k = 1$ and 
thus $x_{k+1} = x_k + \dir{x}_k \in \interior{\X}$. For brevity, denote $\hat{y}_{k+1} = y_k + \dir{y}_k$, and $s_{k+1} = \cons(x_{k+1})$.
Also, using $\min\{\alphaInd{k}^{(1)},\alphaInd{k}^{(2)}\}\ge 1$ and \eqref{eq:bound-slack-variable}, we have
\begin{flalign}\label{eq:Sinv-bound-by-constant}
\|S_k^{-1}\dir{s}_k\|_2 \underle{(a)} \sqrt{\frac{-2\mk}{\mu}} \underle{(b)} \frac{1}{8}
\end{flalign}
where $(a)$ uses \Cref{lem:invS-ds-bound} and $(b)$ uses $-\mk \le \mu / 128$
from $\alphaInd{k}^{(1)} \ge 1$.
Next note that using $S_k \dir{y}_k + Y_k \dir{s}_k + S_k y_k = \mu \ones$ and $S_k y_k=\mu\ones$, 
we have 
\begin{flalign}\label{eq:dir-y-formula}
\dir{y}_k=-Y_k S_k^{-1} \dir{s}_k
\end{flalign}
Thus $\| Y_k^{-1} \dir{y}_k \|_{\infty} \le \| Y_k^{-1} \dir{y}_k \|_2 = \|S_k^{-1}\dir{s}_k\|_2 \le 1/8$ and using
\Cref{eq:Sinv-bound-by-constant}  we conclude $Y_k^{-1} (y_{k} + \dir{y}_k) > 0$ and 
therefore \eqref{eq:first-order-SIP-full:feasible} holds.
The remainder of the proof establishes that \eqref{eq:first-order-SIP-full:centrality} and \eqref{eq:first-order-SIP-full:grad-lag} hold.
 
We first show the complementarity condition \eqref{eq:first-order-SIP-full:centrality} holds. Observe that
{
\footnotesize
\begin{flalign*}
\| \diag(\hat{y}_{k+1}) s_{k+1} - \mu \ones\|_2
&\underle{(a)} \mu \|S_k^{-1} \dir{s}_k\|_2^2 + \frac{\LipGrad}{2} \| y_k \|_{2} (1 + \|S_k^{-1} \dir{s}_k\|_2) \| \dir{x}_k \|_2^2 \\
&\underle{(b)} \frac{\mu}{64} + \frac{9 \LipGrad}{16} \| y_k \|_{2} \| \dir{x}_k \|_2^2 \\
&\underle{(c)} \frac{\mu}{64} + \frac{9 \LipGrad}{16} \| y_k \|_{2} \left(\frac{-\mk}{\eta_1 \mu \sqrt{\|y_k\|_1 + 1}}\right)^2 \\
&\underle{(d)} \frac{\mu}{64} + \frac{\LipGrad}{2} \left(1 + \frac{1}{8} \right) \cdot \frac{\mu}{8\LipGrad} \cdot \frac{\| y_k \|_2}{1 + \| y_k \|_1} \le \frac{\mu}{64} + \frac{9}{128}\mu = \frac{11}{128}\mu < \frac{1}{2}\mu,
\end{flalign*}
}%
where inequality $(a)$ uses \eqref{eq:new-general-comp-bound}, inequality $(b)$ uses \Cref{eq:Sinv-bound-by-constant},
$(c)$ uses \eqref{eq:direction-condition-eta1} with $\varepsilon_k = \mu \sqrt{1 + \| y_k \|_1}$, 
and $(d)$ uses $-\mk \le \frac{\eta_1 \mu^{3/2}}{2\sqrt{2}\,\LipGrad^{1/2}}$ from $\alphaInd{k}^{(2)} \ge 1$.
Since $\|\cdot\|_\infty\le \|\cdot\|_2$, the complementarity condition \eqref{eq:first-order-SIP-full:centrality} holds.
 
Now we show that the gradient of the Lagrangian is small, i.e., \eqref{eq:first-order-SIP-full:grad-lag} holds.
Substituting \Cref{eq:dir-y-formula} and $\dir{s}_k = \grad \cons(x_k) \dir{x}_k$ into $(\grad^2 \barrier(x_k) + \eye \delta_k)\dir{x}_k = -\grad \barrier(x_k)$ gives
\[
\grad_{x} \Lag(x_k, y_k) + \grad_{xx}^2 \Lag(x_k, y_k)^\top \dir{x}_k - \grad \cons(x_k)^\top \dir{y}_k = -\delta_k \dir{x}_k.
\]
Combining this with \Cref{eq:LagError} yields
\begin{flalign*}
\| \delta_k \dir{x}_k + \grad_{x} \Lag(x_{k+1}, \hat{y}_{k+1}) \|
&\le \frac{\LipHess}{2}(\|y_k\|_1 + 1)\|\dir{x}_k\|_2^2 + \LipGrad \|y_k\|_2 \|\dir{x}_k\|_2 \|S_k^{-1}\dir{s}_k\|_2.
\end{flalign*}
Therefore,
\begin{flalign*}
\|\grad_{x}\Lag(x_{k+1}, \hat{y}_{k+1})\|
&\le \|\delta_k \dir{x}_k\| + \frac{\LipHess}{2}(\|y_k\|_1+1)\|\dir{x}_k\|_2^2 + \LipGrad\|y_k\|_2 \|\dir{x}_k\|_2 \|S_k^{-1}\dir{s}_k\|_2 \\
&\underle{(a)} 2\mu\eta_2\sqrt{\|y_k\|_1+1} + \frac{\LipHess}{2}(\|y_k\|_1+1)\|\dir{x}_k\|_2^2 + \LipGrad\|y_k\|_2 \|\dir{x}_k\|_2 \|S_k^{-1}\dir{s}_k\|_2 \\
&\underle{(b)} 2\mu\eta_2\sqrt{\|y_k\|_1+1} + \frac{\LipHess}{2}\left(\frac{-\mk}{\mu\eta_1}\right)^2 + \LipGrad\|y_k\|_2\cdot\frac{2^{1/2}(-\mk)^{3/2}}{\mu^{3/2}\eta_1\sqrt{\|y_k\|_1+1}} \\
&\underle{(d)} 2\mu\eta_2\sqrt{\|y_k\|_1+1} + \frac{\LipHess}{2\eta_1^2}\left(\frac{-\mk}{\mu}\right)^2 + \frac{2^{1/2}\LipGrad\sqrt{\|y_k\|_1+1}}{\eta_1}\left(\frac{-\mk}{\mu}\right)^{3/2} \\
&\underle{(e)} 2\mu\eta_2\sqrt{\|y_k\|_1+1} + \frac{(1-\SufficientConstant)\eta_1\mu}{2} + \frac{(1-\SufficientConstant)\eta_1\mu\sqrt{\|y_k\|_1+1}}{8} \\
&\le \mu\sqrt{\|y_k\|_1+1}\left(2\eta_2 + \frac{5}{8}(1-\SufficientConstant)\eta_1\right) \underle{(f)} \mu\sqrt{7/8}\sqrt{\|y_k\|_1+1}
\underle{(g)} \mu\sqrt{\|\hat{y}_{k+1}\|_1+1}
\end{flalign*}
where inequality $(a)$ uses $\|\dir{x}_k\|_2 \delta_k \le 2\eta_2 \varepsilon_k = 2\eta_2 \mu \sqrt{\|y_k\|_1 + 1}$ which 
holds by \Cref{eq:direction-condition-eta2} and $\barrier(x_k + \dir{x}_k) \le \barrier(x_k) + \SufficientConstant\cdot\Mk(\dir{x}_k)$
since $\alpha_k=1$; $(b)$ uses~\eqref{eq:direction-condition-eta1} from Condition~\ref{condition:search-dir}
and \Cref{lem:invS-ds-bound};
inequality $(d)$ uses $\|y_k\|_2 \le \|y_k\|_1 \le \|y_k\|_1 + 1$;
inequality $(e)$ uses
\[
-\mk \le \min\left\{ \frac{(1-\SufficientConstant)^{1/2}\mu^{3/2}\eta_1^{3/2}}{\LipHess^{1/2}},\; \frac{(1-\SufficientConstant)^{2/3}\mu^{5/3}\eta_1^{4/3}}{(8\sqrt{2})^{2/3}\LipGrad^{2/3}} \right\}
\]
from $\min\{\alphaInd{k}^{(3)}, \alphaInd{k}^{(4)}\} \ge 1$;
inequality $(f)$ uses that $2\eta_2 + \tfrac{5}{8}(1-\SufficientConstant)\eta_1 \le \sqrt{7/8}$; and inequality $(g)$ uses 
$\hat{y}_{k+1} = Y_{k} (\mathbf{1} + Y_{k}^{-1} \dir{y}_k) \ge y_k (1 - \| Y_{k}^{-1} \dir{y}_k \|_{\infty}) \ge (7/8) y_k$
where the last inequality uses \Cref{eq:Sinv-bound-by-constant} and \Cref{eq:dir-y-formula}.
\end{proof}

\Cref{lemBarrierProgress} shows that if $\min_{i \in [7]} \alphaInd{k}^{(i)} < 1$ then we obtain a
 reduction in the barrier function of at least $\Omega(\mu^{5/3})$.
\Cref{lemBarrierTermination} shows that if $\min_{i \in [7]} \alphaInd{k}^{(i)} \ge 1$ then the algorithm terminates.
Combining these two results together using standard telescoping arguments gives \Cref{main-theorem-IPM}.

\begin{theorem}\label{main-theorem-IPM}
Suppose Assumption~\ref{assume-lip-deriv} holds. 
Recall the definition of $\remainder(\mu)$ from \Cref{eq:formula-for-remainder}.
Then, \Cref{alg:IPM} terminates after at most
\begin{flalign*}
K \le 1 + \frac{\barrier(x_1) -\barrier^\star}{\SufficientConstant \backtrackingFactor \mu} \left( \MainConst{} \mu^{-2/3} + \remainder(\mu) \right)
\end{flalign*}
iterations with a $\mu$-approximate first-order SIP, i.e., \Cref{eq:first-order-SIP-full:first-order} holds.
\end{theorem}
\begin{proof}
Let $K$ be the iteration that the algorithm terminates, with $K=\infty$ signifying the algorithm 
does not terminate.
By \Cref{lemBarrierTermination} we deduce that $\min_{i \in [7]} \alphaInd{k}^{(i)} < 1$
for all $k \in [1,K)$.
	Using \Cref{lemBarrierProgress} and telescoping, we get:
	\begin{flalign*}
		\barrier(x_1) -\barrier^\star &\ge \sum_{k=1}^{K-1} (\barrier(x_k) - \barrier(x_{k+1})) \\ 
		&\geq (K-1) \SufficientConstant \backtrackingFactor \left( \MainConst{} \mu^{-5/3} + \mu^{-1} \remainder(\mu) \right)^{-1}.
	\end{flalign*}
	 Rearranging the above inequality for $K$ gives the desired result.
\end{proof}

\section{Convergence analysis of our barrier method with annealing scheme}\label{sec:optimality-guarantees-convexity-regularity-condition}

\newcommand{\InflationFactor}{\kappa}
\newcommand{\doubleBarX}[0]{\bar{\bar{x}}}
\newcommand{\doubleBarY}[0]{\bar{\bar{y}}}

\Cref{alg:IPM} focused on the case that the barrier parameter $\mu$ was fixed.
Our annealing scheme is described in \Cref{alg:Full-IPM} which calls \Cref{alg:IPM} which finds an 
approximate stationary interior point and decreases $\mu$, as per a standard interior point method \cite[Section~19.4]{nocedal2006numerical}.

\newcommand{\muBacktrackingFactor}[0]{\sigma}

\begin{condition}[Properties the new primal-dual iterate must satisfy when we decrease the barrier parameter]
\label{condition:new-x-y}
Suppose that the iterate $(\doubleBarX_{j+1}, \doubleBarY_{j+1})$ and barrier parameter $\mu_{j+1}$ satisfy
\begin{subequations}\label{eq:condition:new-x-y}
\begin{flalign}
\mu_{j+1} &\le \muBacktrackingFactor \mu_j \\
\cons(\doubleBarX_{j+1}) &\ge 0 \\
\| \grad_x \Lag(\doubleBarX_{j+1}, \doubleBarY_{j+1}) \|&\le \InflationFactor \mu_{j+1} \sqrt{1 + \| \doubleBarY_{j+1} \|_1} \label{eq:kappa-bound-gradient-of-Lagrangian} \\
\frac{1}{\InflationFactor} \le & ~ \frac{\cons_i(\doubleBarX_{j+1}) [\doubleBarY_{j+1}]_i}{\mu_{j+1}} \le \InflationFactor \quad \quad  \forall i \in [\NumCon]. \label{eq:complementarity-bound}
\end{flalign}
\end{subequations}
where $\InflationFactor \in (2,\infty)$ and $\muBacktrackingFactor \in [2/\InflationFactor, 1)$ are algorithmic constants. 
\end{condition}

Given a point $(\bar{x}_{j+1}, \bar{y}_{j+1})$ that is a $\mu_j$ approximate stationary interior point, 
i.e., satisfies \Cref{eq:first-order-SIP-full:first-order},
Condition~\ref{condition:new-x-y} can be achieved trivially 
by setting $(\doubleBarX_{j+1}, \doubleBarY_{j+1}) = (\bar{x}_{j+1}, \bar{y}_{j+1})$ and 
$\mu_{j+1} =  \muBacktrackingFactor \mu_{j}$. However, in practice, we will use a more sophisticated 
method, see \Cref{alg:decrease-barrier-parameter}.

\begin{algorithm}[ht!]
	\caption{Line search Newton regularized method applied to the log barrier with decreasing barrier parameter $\mu$.}
	\label{alg:Full-IPM}
	\begin{algorithmic}[1]
		\Function{IPMwithBarrierParameterAnnealing}{$\doubleBarX_1, \mu_1$}
			\State \textbf{Input requirements:} $\doubleBarX_1 \in \interior{\X}$, $\mu_1 \in (0, \infty)$
			\State \textbf{Parameters:} $\SufficientConstant \in (0, 1/2)$, $\eta_1 \in (0, 1/2)$, $\eta_2 \in (\eta_1, 1/2)$, $\InflationFactor \in (2,\infty)$, $\muBacktrackingFactor \in [2/\InflationFactor, 1)$
			\For{$j = 1, \dots, \infty$}
				\State \textbf{\emph{Minimize the log barrier with a fixed barrier parameter:}}
				\State ~$(\bar{x}_{j+1}, \bar{y}_{j+1}) \gets \Call{AdaptiveLineSearchIPM}{\doubleBarX_j, \mu_j}$
				\State \textbf{\emph{Decrease barrier parameter $\mu_{j+1}$ and find new starting point $\doubleBarX_{j+1}$:}}
				\State ~Find $(\doubleBarX_{j+1}, \doubleBarY_{j+1})$ and barrier parameter $\mu_{j+1}$ satisfying Condition~\ref{condition:new-x-y}, e.g., \Cref{alg:method-to-decrease-barrier-parameter}
			\EndFor
			\EndFunction
	\end{algorithmic}
\end{algorithm}

\begin{algorithm}[ht!]
\caption{Practical method to decrease barrier parameter $\mu$ using a standard primal-dual interior point search
direction to select a new iterate so that Condition~\ref{condition:new-x-y} holds.}\label{alg:method-to-decrease-barrier-parameter}
\label{alg:decrease-barrier-parameter}
\begin{algorithmic}[1]
\State \textbf{Input requirements:} $\bar{x}_{j+1} \in \R^{\NumVar}, \bar{y}_{j+1} \in \R^{\NumCon}$, $\mu_j \in (0,\infty)$, and $(\bar{x}_{j+1},\bar{y}_{j+1})$ satisfies \Cref{eq:first-order-SIP-full:first-order}
\State \textbf{Parameter:} $\InflationFactor \in (2,\infty)$, $\muBacktrackingFactor \in [2/\InflationFactor, 1)$
\State $(x^{\mathrm{prev}},y^{\mathrm{prev}},\mu^{\mathrm{prev}}) \gets
\left(\bar{x}_{j+1},\bar{y}_{j+1}, \muBacktrackingFactor \mu_j\right)$
\For{$t = 0,1,2,\ldots$}
	\State $\hat{\mu}_t \gets \muBacktrackingFactor^{t+1}\mu_j$
	\State Calculate search direction $(\dir{x},\dir{y})$ via a typical primal-dual interior point search direction:
	\begin{subequations}\label{eq:typical-primal-dual-interior-point-search-direction}
\begin{flalign}
\bar{S}_{j+1} = \diag(\cons(\bar{x}_{j+1})) \quad \bar{Y}_{j+1} &= \diag(\bar{y}_{j+1}) \\
\grad_{xx}^2 \Lag(\bar{x}_{j+1},\bar{y}_{j+1}) \dir{x} - \grad \cons(\bar{x}_{j+1})^\top \dir{y} &= -\grad_x \Lag(\bar{x}_{j+1},\bar{y}_{j+1})  \\
\grad \cons(\bar{x}_{j+1}) \dir{x} - \dir{s} &= \zeros \\
\bar{S}_{j+1}  \dir{y} + \bar{Y}_{j+1} \dir{s} &= \hat{\mu}_t \ones - \bar{S}_{j+1} \bar{y}_{j+1}.
\end{flalign}
\end{subequations}
	\State $(x^{\mathrm{trial}},y^{\mathrm{trial}}) \gets
	(\bar{x}_{j+1}+\dir{x},\bar{y}_{j+1}+\dir{y})$
	\If{$(x^{\mathrm{trial}},y^{\mathrm{trial}},\hat{\mu}_t)$ does not satisfy Condition~\ref{condition:new-x-y}}
		\State \Return $(x^{\mathrm{prev}},y^{\mathrm{prev}},\mu^{\mathrm{prev}})$
	\EndIf
	\State $(x^{\mathrm{prev}},y^{\mathrm{prev}},\mu^{\mathrm{prev}}) \gets
	(x^{\mathrm{trial}},y^{\mathrm{trial}},\hat{\mu}_t)$
\EndFor
\end{algorithmic}
\end{algorithm}

We provide guarantees for \Cref{alg:decrease-barrier-parameter} in terms of the total number of regularized Newton steps taken by the method
for the unconstrained case (\Cref{thm:unconstrained-case}) and constrained case (\Cref{thm:constrained-case}) respectively. 
One regularized Newton step in \Cref{alg:Full-IPM} corresponds to an iteration of $\Call{AdaptiveLineSearchIPM}{}$.
Thus, we can use the total number of regularized Newton steps of \Cref{alg:Full-IPM} 
as a measure of its computational complexity.

To prove \Cref{thm:unconstrained-case} and 
\Cref{thm:constrained-case} we use a standard argument that after solving
a barrier subproblem we can 
upper bound the suboptimality of 
a new barrier subproblem which 
we can use to reduce our estimate of the number of Newton steps required to solve the next subproblem.
This is very similar to standard restart arguments \cite{o2015adaptive} which are widespread in 
the optimization literature.
In particular, we use \Cref{thm:bounded-initial-suboptimality-unconstrained-case} and \Cref{main-theorem-IPM} to establish 
\Cref{assume:small-mu-gives-good-suboptimality} in the unconstrained and constrained setting respectively,
and then employ \Cref{lem:generic-ipm-guarantee}, which is a generic result about
the number of Newton steps taken by \Cref{alg:Full-IPM} when \Cref{assume:small-mu-gives-good-suboptimality} holds
to provide bounds on the number of iterations our method takes to find an 
$\epsilon$-approximately optimal solution in the unconstrained 
(\Cref{thm:unconstrained-case}) and 
constrained case (\Cref{thm:constrained-case}).

\newcommand{\Rate}[0]{\theta}
\newcommand{\Dconst}[0]{W}

\begin{assumption}\label{assume:small-mu-gives-good-suboptimality}
Suppose that, for any $x \in \interior{\X}$ and $\mu \in (0,\infty)$, the method \\ $\Call{AdaptiveLineSearchIPM}{x, \mu}$ requires at most $1 + (\barrier(x) - \barrier^\star) \mu^{-1} \Rate(\mu)$
iterations to return a point satisfying \Cref{eq:first-order-SIP-full:first-order} where $\Rate : (0,\infty) \rightarrow (0,\infty)$ 
is a nonincreasing function. 
Additionally, assume that there exists a constant $\Dconst > 0$
such that
$\barrier[j+1](\doubleBarX_{j+1}) - \barrier[j+1]^\star \le \Dconst \mu_{j+1}$
and
$f(\doubleBarX_{j+1}) - f_\star \le \Dconst \mu_{j+1}$
for all $j \in \N$.
\end{assumption}

\newcommand{\Jhat}[0]{J}

\begin{lemma}\label{lem:generic-ipm-guarantee}
Suppose \Cref{assume:small-mu-gives-good-suboptimality} holds.
Then, for any $\epsilon \in (0,\infty)$,
there exists an iterate $\doubleBarX_{\Jhat+1}$ of \Cref{alg:Full-IPM} with $\Jhat \le 2 + \log_{1/\muBacktrackingFactor}( \Dconst \mu_1 / \epsilon + 1 / \muBacktrackingFactor)$ such that 
$f(\doubleBarX_{\Jhat+1}) - f_\star \le \epsilon$ and the total number of regularized Newton steps used by \Cref{alg:Full-IPM} to compute $\doubleBarX_{\Jhat+1}$ is at most 
\[
\Jhat +  \Dconst \sum_{j=0}^{\Jhat-2} \Rate\left( \frac{\epsilon}{\Dconst \muBacktrackingFactor^{j}} \right) + (\barrier[1](\doubleBarX_1) - \barrier[1]^\star) \mu_1^{-1} \Rate(\mu_1).
\]
\end{lemma}

\begin{proof}
First, observe that the number of iterations until $\Call{AdaptiveLineSearchIPM}{\doubleBarX_1, \mu_1}$ terminates is at 
most $1 + (\barrier[1](\doubleBarX_1) - \barrier[1]^\star) \mu_1^{-1} \Rate(\mu_1)$.
If $\epsilon / \Dconst \ge \mu_1$ then the result holds with $\Jhat = 1$ because 
$\mu_2 \le \muBacktrackingFactor \mu_1 < \epsilon / \Dconst$ and 
$f(\doubleBarX_{2}) - f_\star \le \Dconst \mu_{2}$; so consider the case that 
$\epsilon / \Dconst < \mu_1$.
Since $\mu_{j+1} \le \muBacktrackingFactor \mu_j$ we deduce that there is an index 
$\Jhat \le 2 + \log_{1/\muBacktrackingFactor}( \Dconst \mu_1 / \epsilon + 1 / \muBacktrackingFactor)$ such that $\mu_{\Jhat+1} < \epsilon / \Dconst \le \mu_{\Jhat}$.
Thus, the total number of regularized Newton steps of \Cref{alg:Full-IPM} across iterations $j=2, \dots, \Jhat$ of \Cref{alg:Full-IPM}  is at most
\begin{flalign*}
\sum_{j=2}^{\Jhat} 1 + \frac{(\barrier[j](\doubleBarX_j) - \barrier[j]^\star) \Rate(\mu_j)}{\mu_j} 
\underle{(i)} \sum_{j=2}^{\Jhat} 1 + \Dconst \Rate(\mu_j) \underle{(ii)}  \sum_{j=0}^{\Jhat-2} 1 + \Dconst  \Rate\left( \frac{\mu_{\Jhat}}{\muBacktrackingFactor^{j}} \right) \underle{(iii)} \Jhat - 1 + \Dconst \sum_{j=0}^{\Jhat-2} \Rate\left( \frac{\epsilon}{\Dconst \muBacktrackingFactor^{j}} \right)
\end{flalign*}
where $(i)$ uses $\barrier[j](\doubleBarX_j) - \barrier[j]^\star \le \Dconst \mu_j$ for all $j \ge 2$, $(ii)$ uses $\mu_{j+1} \le \muBacktrackingFactor \mu_j$ and the fact that $\Rate$ is nonincreasing, and $(iii)$ uses that $ \epsilon / \Dconst \le \mu_{\Jhat}$ and the fact that $\Rate$ is nonincreasing. Adding the bound for the first call to $\Call{AdaptiveLineSearchIPM}{}$ gives the claimed bound on the regularized Newton steps. Furthermore, by $f(\doubleBarX_{j+1}) - f_\star \le \Dconst \mu_{j+1}$ and $\mu_{\Jhat+1} < \epsilon / \Dconst$, we get
$f(\doubleBarX_{\Jhat+1}) - f_\star \le \Dconst \mu_{\Jhat+1} < \epsilon$.
\end{proof}

\Cref{thm:unconstrained-case} provides a convergence guarantee for 
our algorithm in the unconstrained setting, showing that 
it matches the standard convergence $O(\epsilon^{-1/2})$ bound \cite{nesterov2006cubic}. This is a better rate
than the $O(\epsilon^{-2/3})$ result that we prove in the 
constrained setting.

\begin{theorem}[Unconstrained case]\label{thm:unconstrained-case}
Suppose that the problem is unconstrained, i.e., $\barrier(x) = f(x)$ and $\grad^2 f$ is $\LipHess$-Lipschitz. 
Let $D := \diam(\{ x \in \R^{\NumVar} : f(x) \le f(\doubleBarX_1) \})$
and let $\hat{C} := \SufficientConstant \cdot \eta_1 \cdot \min\{\backtrackingFactor \cdot \sqrt{6(1- \SufficientConstant) \eta_1} , \sqrt{2(1-2\eta_2)}\}$ 
be a problem-independent constant.
Then, for any $\epsilon \in (0,\infty)$,
there exists an iterate $\doubleBarX_{\Jhat+1}$ of \Cref{alg:Full-IPM} with $\Jhat \le 2 + \log_{1/\muBacktrackingFactor}( \InflationFactor D \mu_1 / \epsilon + 1 / \muBacktrackingFactor)$ such that 
$f(\doubleBarX_{\Jhat+1}) - f_\star \le \epsilon$ and the total regularized Newton steps \Cref{alg:Full-IPM} uses to compute $\doubleBarX_{\Jhat+1}$ is at most 
\[
\frac{\hat{C}^{-1}\InflationFactor^{3/2}}{1-\sqrt{\muBacktrackingFactor}} \cdot \frac{\LipHess^{1/2} D^{3/2}}{\epsilon^{1/2}} + 
2+\log_{1/\muBacktrackingFactor}( \InflationFactor D \mu_1 / \epsilon + 1 / \muBacktrackingFactor) +
\frac{\hat{C} ^ {-1} (f(\doubleBarX_1) - f_\star)\LipHess^{1/2}}{\mu_1^{3/2}}.
\]
\end{theorem}

\begin{proof}
We will establish the result using \Cref{lem:generic-ipm-guarantee}. First, note that in the unconstrained case, Condition~\ref{condition:new-x-y} implies
$\|\grad f(\doubleBarX_{j+1})\|_2 \le \InflationFactor \mu_{j+1}$.
Since the iterates remain in the initial sublevel set, convexity gives
$f(\doubleBarX_{j+1}) - f_\star 
\le \|\grad f(\doubleBarX_{j+1})\|_2 \|\doubleBarX_{j+1}-x^\star\|_2
\le \InflationFactor D \mu_{j+1}$.
Thus, \Cref{assume:small-mu-gives-good-suboptimality} holds with $\Dconst=\InflationFactor D$. From \Cref{thm:bounded-initial-suboptimality-unconstrained-case} we have 
$\Rate(\mu) = \hat{C} ^ {-1}  \LipHess ^ {1/2} \mu ^ {-1/2}$.
Thus, by the formula for the sum of an infinite geometric series,
\begin{flalign*}
\Dconst\sum_{j=0}^{\Jhat-2} \Rate\left( \frac{\epsilon}{\Dconst \muBacktrackingFactor^{j}} \right) &=
\Dconst\sum_{j=0}^{\Jhat-2}\hat{C}^{-1}\LipHess^{1/2}\left( \frac{\epsilon}{\Dconst \muBacktrackingFactor^{j}} \right)^{-1/2} =
\hat{C}^{-1}\LipHess^{1/2}\Dconst^{3/2}\epsilon^{-1/2}\sum_{j=0}^{\Jhat-2}\muBacktrackingFactor^{j/2} \\
&\le
\frac{\hat{C}^{-1}\LipHess^{1/2}\Dconst^{3/2}}{(1-\sqrt{\muBacktrackingFactor})\epsilon^{1/2}} =
\frac{\hat{C}^{-1}\InflationFactor^{3/2}\LipHess^{1/2}D^{3/2}}{(1-\sqrt{\muBacktrackingFactor})\epsilon^{1/2}}.
\end{flalign*}
Moreover, by \Cref{lem:generic-ipm-guarantee},
$\Jhat \le 
2+\log_{1/\muBacktrackingFactor}( \InflationFactor D \mu_1 / \epsilon + 1 / \muBacktrackingFactor)$.
Finally, $(f(\doubleBarX_1)-f_\star)\mu_1^{-1}\Rate(\mu_1)
=
\hat{C} ^ {-1} (f(\doubleBarX_1) - f_\star)\LipHess^{1/2}\mu_1^{-3/2}$.
Combining these three bounds with \Cref{lem:generic-ipm-guarantee} gives the result.
\end{proof}

\newcommand{\SlaterSlack}{t}
\newcommand{\LowOrderTerm}{\ell}


Before proving our convergence guarantee in the constrained setting (\Cref{thm:constrained-case}) we prove \Cref{lem:bound-y-norm} which provides a bound on the
norm of the dual variables when Condition~\ref{condition:new-x-y} holds. This is useful 
to establish bounds on the suboptimality needed to show \Cref{assume:small-mu-gives-good-suboptimality}.
To prove this result we use standard arguments \cite{haeser2021behavior} to leverage the existence of
a Slater point, which is well-known to ensure that the dual variables are bounded \cite{bertsekas1999note}.

\begin{lemma}[Bound on the size of the dual multipliers]\label{lem:bound-y-norm}
Suppose \Cref{assume-lip-deriv} holds. Define $D_{\X} = \diam(\X)$ and $\SlaterSlack := \min_{i\in[\NumCon]}\cons_i(\hat{x})>0$, i.e., $\hat{x}$ is a Slater point. 
Define
\begin{flalign}\label{eq:bound-y-norm}
\lambda
:=
\frac{
\InflationFactor D_{\X}\mu_1
+
\sqrt{
\InflationFactor^2D_{\X}^2\mu_1^2
+
4\SlaterSlack
\left(
f(\hat{x})-f_\star
+
\InflationFactor\NumCon\mu_1
+
\SlaterSlack
\right)
}
}{
2\SlaterSlack
}.
\end{flalign}
If Condition~\ref{condition:new-x-y} holds then $\sqrt{1 + \| \doubleBarY_{j+1} \|_1} \le \lambda$.
\end{lemma}

\begin{proof}
Observe that
\begin{flalign*}
\SlaterSlack\|\doubleBarY_{j+1}\|_1
&\underle{(i)} \sum_{i=1}^{\NumCon}[\doubleBarY_{j+1}]_i\cons_i(\hat{x}) \underle{(ii)}
\sum_{i=1}^{\NumCon}[\doubleBarY_{j+1}]_i\cons_i(\doubleBarX_{j+1})
+
\left(\grad\cons(\doubleBarX_{j+1})^\top \doubleBarY_{j+1}\right)^\top
(\hat{x}-\doubleBarX_{j+1}) \\
&\underle{(iii)}
\sum_{i=1}^{\NumCon}[\doubleBarY_{j+1}]_i\cons_i(\doubleBarX_{j+1})
+
\grad f(\doubleBarX_{j+1})^\top(\hat{x}-\doubleBarX_{j+1}) -
\grad_x\Lag(\doubleBarX_{j+1},\doubleBarY_{j+1})^\top
(\hat{x}-\doubleBarX_{j+1}) \\
&\underle{(iv)}
\InflationFactor\NumCon\mu_{j+1}
+
f(\hat{x})-f(\doubleBarX_{j+1})
+
D_{\X}\|\grad_x\Lag(\doubleBarX_{j+1},\doubleBarY_{j+1})\|_2 \\
&\underle{(v)}
\InflationFactor\NumCon\mu_1
+
f(\hat{x})-f_\star
+
\InflationFactor D_{\X}\mu_1
\sqrt{1+\|\doubleBarY_{j+1}\|_1}.
\end{flalign*}
where $(i)$ uses the Slater condition and $[\doubleBarY_{j+1}]_i>0$ for all $i\in[\NumCon]$, 
$(ii)$ uses concavity of each $\cons_i$, $[\doubleBarY_{j+1}]_i\ge 0$,
$(iii)$ uses the identity
$\grad_x\Lag(\doubleBarX_{j+1},\doubleBarY_{j+1})
=
\grad f(\doubleBarX_{j+1})
-
\grad\cons(\doubleBarX_{j+1})^\top\doubleBarY_{j+1}$, $(iv)$ uses the complementarity upper bound 
in \Cref{eq:condition:new-x-y}, convexity of $f$, and
$\|\hat{x}-\doubleBarX_{j+1}\|_2 \le D_{\X}$, and $(v)$ uses $\mu_{j+1}\le \mu_1$,
$f(\doubleBarX_{j+1})\ge f_\star$, and the termination criterion for the gradient of the Lagrangian given in 
\Cref{eq:kappa-bound-gradient-of-Lagrangian}.
Rearranging the previous display gives
\[
\SlaterSlack\left(1+\|\doubleBarY_{j+1}\|_1\right)
-
\InflationFactor D_{\X}\mu_1
\sqrt{1+\|\doubleBarY_{j+1}\|_1}
-
\left(
f(\hat{x})-f_\star
+
\InflationFactor\NumCon\mu_1
+
\SlaterSlack
\right)
\le 0.
\]
Solving this quadratic inequality in
\(\sqrt{1+\|\doubleBarY_{j+1}\|_1}\) gives
\[
\sqrt{1+\|\doubleBarY_{j+1}\|_1}
\le
\frac{
\InflationFactor D_{\X}\mu_1
+
\sqrt{
\InflationFactor^2D_{\X}^2\mu_1^2
+
4\SlaterSlack
\left(
f(\hat{x})-f_\star
+
\InflationFactor\NumCon\mu_1
+
\SlaterSlack
\right)
}
}{
2\SlaterSlack
}
=
\lambda,
\]
which proves the result.
\end{proof}

\Cref{thm:constrained-case} is our main result which shows 
that the total number of regularized Newton steps used in \Cref{alg:Full-IPM} to find 
an $\epsilon$-optimal solution is $O(\epsilon^{-2/3})$.
It is a consequence of \Cref{main-theorem-IPM} which gives a
bound on the number of Newton steps for each call to $\Call{AdaptiveLineSearchIPM}{}$ allowing us to establish the premise of 
\Cref{lem:generic-ipm-guarantee} using \Cref{lem:bound-y-norm} and thus prove the result.

\begin{theorem}[Constrained case]\label{thm:constrained-case}
Suppose \Cref{assume-lip-deriv} holds. Define $D_{\X} = \diam(\X) < \infty$.
Let $\Dconst := \InflationFactor\left(\NumCon+D_{\X} \lambda \right)$ where $\lambda$ is defined in \Cref{eq:bound-y-norm}, i.e.,
there exists a Slater point.
Then, for any $\epsilon\in(0,\infty)$, there exists an iterate
$\doubleBarX_{\Jhat+1}$ of \Cref{alg:Full-IPM} with
$\Jhat \le 2+\log_{1/\muBacktrackingFactor} \left(\frac{\Dconst\mu_1}{\epsilon} + \frac{1}{\muBacktrackingFactor}\right)$
such that $f(\doubleBarX_{\Jhat+1})-f_\star\le \epsilon$,
and the total number of regularized Newton steps used in \Cref{alg:Full-IPM} to compute
$\doubleBarX_{\Jhat+1}$ is at most
\[
\frac{8 \LipGrad^{2/3} \Dconst^{5/3}}{(1 - \SufficientConstant)^{2/3} \eta_1^{4/3} \SufficientConstant\backtrackingFactor(1-\muBacktrackingFactor^{2/3})} \cdot
\epsilon^{-2/3}
+
\ell(\epsilon),
\]
where $\ell(\epsilon) = O(\log(1/\epsilon)\epsilon^{-3/5})$, in particular,
\[
\ell(\epsilon)
:=
\frac{\Dconst\Jhat}{\SufficientConstant\backtrackingFactor}
\remainder(\Dconst^{-1}\epsilon)
+
\Jhat
+
\frac{\barrier[1](\doubleBarX_1)-\barrier[1]^\star}
{\SufficientConstant\backtrackingFactor\mu_1}
\left(
    \MainConst{}\mu_1^{-2/3}
    +
    \remainder(\mu_1)
\right),
\]
and $\remainder(\mu)$, which is $O(\mu^{-3/5})$, is defined in  \Cref{eq:formula-for-remainder}.
\end{theorem}

\newcommand{\xMuStar}{x_{\mu_{j+1}}^\star}

\begin{proof}
We establish the result using \Cref{lem:generic-ipm-guarantee}. Let
$\xMuStar\in\argmin_{x\in \X}\barrier[j+1](x)$.
Observe that for all $u, v > 0$ we have
$u (1-v)+\log(v)\le \max\{1/u, u\}$, therefore
setting $u = \frac{[\doubleBarY_{j+1}]_i
\cons_i(\doubleBarX_{j+1})
}{\mu_{j+1}}$ and $v = \frac{\cons_i(\xMuStar)}{\cons_i(\doubleBarX_{j+1})}$ and 
employing $\max\{1/u, u\} \le \kappa$ from \Cref{eq:complementarity-bound} we get
\begin{flalign}\label{eq:bound-y-mu-plus-log}
\sum_{i=1}^{\NumCon}
\frac{[\doubleBarY_{j+1}]_i
(
\cons_i(\doubleBarX_{j+1})-\cons_i(\xMuStar)
)}{\mu_{j+1}} -
\sum_{i=1}^{\NumCon}
\log(
\frac{\cons_i(\doubleBarX_{j+1})}{\cons_i(\xMuStar)}  
) \le \InflationFactor\NumCon.
\end{flalign}
Then,
\begin{flalign*}
\barrier[j+1](\doubleBarX_{j+1})-\barrier[j+1]^\star &\underle{(i)}
\grad f(\doubleBarX_{j+1})^\top
(\doubleBarX_{j+1}-\xMuStar)
-
\mu_{j+1}
\sum_{i=1}^{\NumCon}
\log\!\left(
\frac{\cons_i(\doubleBarX_{j+1})}{\cons_i(\xMuStar)}
\right)
\\
&\underle{(ii)}
\grad_x\Lag(\doubleBarX_{j+1},\doubleBarY_{j+1})^\top (\doubleBarX_{j+1}-\xMuStar)
+
\InflationFactor\NumCon\mu_{j+1}
\\
&\underle{(iii)}
D_{\X}\|\grad_x\Lag(\doubleBarX_{j+1},\doubleBarY_{j+1})\|_2
+
\InflationFactor\NumCon\mu_{j+1}
\\
&\underle{(iv)}
\InflationFactor D_{\X}\lambda\mu_{j+1}
+
\InflationFactor\NumCon\mu_{j+1}
=
\Dconst\mu_{j+1}.
\end{flalign*}
Here, $(i)$ uses the definition of $\barrier$ and convexity of $f$; $(ii)$
uses $\grad f=\grad_x\Lag+\grad\cons^\top\doubleBarY_{j+1}$, 
concavity of each $\cons_i$ and \Cref{eq:bound-y-mu-plus-log}; $(iii)$ uses
Cauchy-Schwarz and $D_{\X} = \diam(\X)$; and $(iv)$ uses \Cref{eq:kappa-bound-gradient-of-Lagrangian}.
Similarly, let $x^\star\in\argmin_{x\in\X} f(x)$. Then,
\begin{flalign*}
f(\doubleBarX_{j+1})-f_\star
&\underle{(i)}
\grad f(\doubleBarX_{j+1})^\top(\doubleBarX_{j+1}-x^\star)
\underle{(ii)}
D_{\X}\|\grad_x\Lag(\doubleBarX_{j+1},\doubleBarY_{j+1})\|_2
+
\sum_{i=1}^{\NumCon}
[\doubleBarY_{j+1}]_i (\cons_i(\doubleBarX_{j+1}) - \cons_i(x_\star))
\\
&\underle{(iii)}
\InflationFactor D_{\X}\lambda\mu_{j+1}
+
\InflationFactor\NumCon\mu_{j+1}
=
\Dconst\mu_{j+1}.
\end{flalign*}
Here, $(i)$ uses convexity of $f$, $(ii)$ uses
$\grad f=\grad_x\Lag+\grad\cons^\top\doubleBarY_{j+1}$,
concavity of each $\cons_i$, $[\doubleBarY_{j+1}]_i\ge 0$, and
$\|\doubleBarX_{j+1}-x^\star\|\le D_{\X}$; and $(iii)$ uses
$\cons_i(x_{\star}) \ge 0$, 
Condition~\ref{condition:new-x-y} and
$\sqrt{1+\|\doubleBarY_{j+1}\|_1}\le\lambda$ from \Cref{lem:bound-y-norm}. Thus, \Cref{assume:small-mu-gives-good-suboptimality} holds with this value of
$\Dconst$.

By \Cref{main-theorem-IPM},
$\Rate(\mu)=
(\SufficientConstant\backtrackingFactor)^{-1}
\left(\MainConst{}\mu^{-2/3}+\remainder(\mu)\right)$.
Therefore, \Cref{lem:generic-ipm-guarantee} gives an iterate
$\doubleBarX_{\Jhat+1}$ with
$\Jhat \le
2+\log_{1/\muBacktrackingFactor}
\left(
\Dconst\mu_1/\epsilon
+
1/\muBacktrackingFactor
\right)$
such that $f(\doubleBarX_{\Jhat+1})-f_\star\le\epsilon$, and the total number of 
regularized Newton steps is at most
\[
\Jhat
+
\Dconst
\sum_{j=0}^{\Jhat-2}
\Rate\left(
\frac{\epsilon}{\Dconst\muBacktrackingFactor^j}
\right)
+
(\barrier[1](\doubleBarX_1)-\barrier[1]^\star)\mu_1^{-1}\Rate(\mu_1).
\]
Moreover,
\begin{flalign*}
\Dconst
\sum_{j=0}^{\Jhat-2}
\Rate\left(
\frac{\epsilon}{\Dconst\muBacktrackingFactor^j}
\right)
&=
\frac{\Dconst}{\SufficientConstant\backtrackingFactor}
\sum_{j=0}^{\Jhat-2}
\left[
\MainConst{}
\left(
\frac{\epsilon}{\Dconst\muBacktrackingFactor^j}
\right)^{-2/3}
+
\remainder\left(
\frac{\epsilon}{\Dconst\muBacktrackingFactor^j}
\right)
\right]
\\
&\underle{(i)}
\frac{\MainConst{}\Dconst^{5/3}}
{\SufficientConstant\backtrackingFactor}
\epsilon^{-2/3}
\sum_{j=0}^{\infty}\muBacktrackingFactor^{2j/3}
+
\frac{\Dconst\Jhat}{\SufficientConstant\backtrackingFactor}
\remainder(\Dconst^{-1}\epsilon)
\\
&=
\frac{\MainConst{}\Dconst^{5/3}}
{\SufficientConstant\backtrackingFactor(1-\muBacktrackingFactor^{2/3})}
\epsilon^{-2/3}
+
\frac{\Dconst\Jhat}{\SufficientConstant\backtrackingFactor}
\remainder(\Dconst^{-1}\epsilon),
\end{flalign*}
where $(i)$ uses that $\remainder$ is nonincreasing and
$\epsilon/(\Dconst\muBacktrackingFactor^j)\ge \epsilon/\Dconst$. Finally,
\[
(\barrier[1](\doubleBarX_1)-\barrier[1]^\star)\mu_1^{-1}\Rate(\mu_1)
=
\frac{\barrier[1](\doubleBarX_1)-\barrier[1]^\star}
{\SufficientConstant\backtrackingFactor\mu_1}
\left(
\MainConst{}\mu_1^{-2/3}
+
\remainder(\mu_1)
\right).
\]
Combining the last three displays proves the result.
\end{proof}

\section{Numerical results}\label{sec:numerical-results}


This section compares our method against IPOPT \cite{wachter2006implementation} and 
\edit{variants} of our interior point method with
our regularized Newton method with line search 
replaced by ARC \cite{cartis2011adaptiveI} \edit{ and AdaN \cite{mishchenko2023regularized}}. 
The latter methods allows us to measure the effectiveness
of our regularized Newton method with line search
for solving the barrier subproblem. 
Our method is competitive with IPOPT, while generally outperforming the variants of our interior point method which use ARC or AdaN as the subproblem solver.

\paragraph{Computing resources} The experiments are performed in a single-threaded environment on a Linux virtual machine that has an Intel(R) Core(TM) i5-3470 CPU 3.20GHz and 16 GB RAM.

\paragraph{Default parameters.}\label{par:default-parameters}
Unless stated otherwise, all runs of our method and ARC use the parameter
values listed in Table~\ref{tab:params}.
The initial barrier parameter is $\mu_1 = 10^{-3}$ and the outer loop
terminates when $\mu_j \leq \StoppingTolerance = 10^{-6}$.
The maximum number of iterations is $100,000$ and the wall-clock
time limit is 3,600 seconds per instance.
\begin{table}[h]
  \centering
  \caption{Default algorithm parameters.}
  \label{tab:params}
  \small
  \begin{tabular}{llll}
    \toprule
    Parameter & Symbol & Value & Introduced in \\
    \midrule
    Sufficient decrease (I) & $\eta_1$ & $2.0\times 10^{-1}$ & \Cref{condition:search-dir} \\
    Sufficient decrease (II) & $\eta_2$ & $2.5\times 10^{-1}$ & \Cref{condition:search-dir} \\
    Armijo constant & $\SufficientConstant$ & $10^{-4}$ & \Cref{step-size-condition} \\
    Backtracking factor & $\backtrackingFactor$ & $5.0\times 10^{-1}$ & \Cref{step-size-condition,alg:compute-step-size} \\
    Inflation factor & $\kappa$ & $8$ & \Cref{condition:new-x-y} \\
    Initial barrier parameter & $\mu_1$ & $10^{-3}$ & \Cref{alg:Full-IPM} \\
    Barrier reduction factor & $\sigma$ & $2.5\times 10^{-1}$ & \Cref{condition:new-x-y,alg:decrease-barrier-parameter} \\
    Stopping tolerance & $\StoppingTolerance$ & $10^{-6}$ & \Cref{par:default-parameters} \\
    Regularization cold start scale & $\ColdStartScale$ & $5.0\times 10^{-1}$ & \Cref{rem:actual-search-direction-implementation} \\
    Regularization full step shrinkage & $\RegularizationShrink$ & $2.0\times 10^{-1}$ & \Cref{rem:actual-search-direction-implementation} \\
    Line search target reduction & $\StartingTargetReduction$ & $1$ & \Cref{prop:line-search-iteration-bound} \\
    \bottomrule
  \end{tabular}
\end{table}

\paragraph{Datasets}

The problem instances draw from standard benchmark datasets, including binary classification tasks from LIBSVM (\url{http://www.csie.ntu.edu.tw/~cjlin/libsvm}) and public email corpora.

\begin{itemize}
	\item \textbf{w8a}: 49,749 samples, 300 sparse features. Labels are $\pm 1$ (web page classification).
	\item \textbf{a9a}: 32,561 samples, 123 features (Adult census-income). Approximately 24\% of labels are positive.
	\item \textbf{Phishing websites}: The LIBSVM phishing dataset, represented by 68 website features and binary labels.
	\item \textbf{SpamAssassin}: A public email corpus categorized by collection period and classification difficulty, tokenized into a 500-word vocabulary.
\end{itemize}

\paragraph{Test problems}

We evaluate our method on a small set of nonlinear convex optimization problems. These specific formulations were selected because they are not naturally or easily reformulatable as standard conic programs (e.g., semidefinite or second-order cone constraints), making general nonlinear optimization an attractive approach. Furthermore, for each of these formulations, it is straightforward to identify a strictly feasible starting point, which is required for our method.

\begin{itemize}
\item 
\textbf{LSE\_Hard.}
A constrained log-sum-exp regression on the w8a dataset. The problem is formulated as follows:
\begin{equation*}
	\begin{aligned}
		\min_{x \in \mathbb{R}^d} \quad & \rho \log \left( \sum_{i=1}^{n} \exp{\left( \frac{a_i^\top x - b_i}{\rho} \right)} \right) \\
		\text{s.t.} \quad & a_j^\top x \leq b_j, \quad j = n+1, \ldots, 2n,
	\end{aligned}
\end{equation*}
where $a_1,\ldots,a_{2n} \in \mathbb{R}^d$ are vectors and $\rho, b_1,\ldots,b_{2n}$ are scalars. The objective serves as a smooth approximation of the maximum function $\max \{a_1^\top x - b_1, \ldots, a_n^\top x - b_n\}$, with the scalar $\rho$ controlling the tightness of the approximation. We utilize the same foundational log-sum-exp formulation evaluated in \cite{mishchenko2023regularized}; however, while their formulation was strictly for unconstrained optimization, we modified the problem to include constraints by partitioning the dataset such that half of the samples define the objective, and the remaining half form explicit linear constraints. For our experiments, we set $d=300$, $n=500$ (resulting in 500 objective terms and 500 constraints), and $\rho = 0.01$.

\item 
\textbf{NP\_Adult.}
A Neyman-Pearson (NP) classification problem \cite{rigollet2011neyman} on the a9a dataset with training samples partitioned into positive and negative classes:
\[
\min_{w} \;
\frac{1}{n_+}\sum_{y_i=+1}
\log\!\bigl(1+e^{-w^\top x_i}\bigr)
\quad \text{s.t.} \quad
\frac{1}{n_-}\sum_{y_i=-1}
\log\!\bigl(1+e^{+w^\top x_i}\bigr) \leq \alpha, \quad
\|w\|_2^2 \leq R^2,
\]
with $\alpha = 0.75$ and $R = 3$. This objective directly addresses asymmetric costs by minimizing the false-negative rate (FNR) surrogate subject to a strict false-positive rate (FPR) budget and an $\ell_2$-norm ball regularizer.

\item 
\textbf{Robust Empirical Risk Minimization (ERM) with Feature Partitions.}
Robust ERM instances \cite{qian2019robust} model scenarios where training data is drawn from heterogeneous subgroups, formulated as a minimax problem over the logistic loss to identify a single shared classifier that guarantees uniformly low loss across all groups. 
There are three variants of this problem:

\begin{itemize}
	\item \textbf{RobustERM\_Spam.}
	Uses the SpamAssassin public corpus partitioned into $K=4$ email-type groups based on their directory source: early spam (collected in 2002), later spam (collected in 2003), easy ham\footnote{In the spam-filtering community, ham denotes legitimate, non-spam email.}, and hard ham. By treating the degree of classification difficulty and collection period as unobserved domains at runtime, this problem tests the method's ability to handle natural distribution shifts in textual data represented by a 500-word vocabulary space.
	
	\item \textbf{RobustERM\_Phishing.}
	Evaluated on the LIBSVM phishing website dataset, where samples are partitioned into $K=4$ groups using two website-feature indicators selected to match the expected prevalences of SSL-state and URL-anchor attributes. These grouping features are removed from the optimization feature matrix.
	
	\item \textbf{RobustERM\_Education.}
	A robust ERM formulation on the a9a census dataset, partitioned into $K=4$ groups representing distinct education levels (Bachelors, Some-college, High School graduate, and a residual). The education features are omitted from the regression. 
\end{itemize}
\end{itemize}

\paragraph{Results}
Table~\ref{tab:results} reports
our results.
Overall, our method is faster than
IPOPT on three of five instances, although notably
slower on \texttt{RobustERM\_Spam}. 

Both AdaN and ARC are consistently  
slower and use many more factorizations our method (with the exception of AdaN on 
\texttt{LSE\_Hard}).
The poor performance of these methods is not unexpected.
In particular, these methods were designed for objectives with Lipschitz-continuous Hessians. 
They maintain a regularization parameter that represents an estimate of the local  
Lipschitz constant of the Hessian. The log barrier Hessian is not Lipschitz continuous, so this local constant 
varies widely throughout the methods trajectories. We believe that repeatedly adapting this 
regularization parameter to these variations is triggering their large number of factorizations.

\ifARXIV
\begin{table}[ht]
	\centering
	\caption{Comparison of the number of evaluations of the objective function, \(f\), constraints, \(\cons\), objective gradient, \(\grad f\), and Hessian of the Lagrangian, \(\grad^2_{xx}\Lag\), linear system factorizations, and wall-clock time across solvers.}
	\label{tab:results}
	\resizebox{\textwidth}{!}{%
		\begin{tabular}{@{}llrrrrrr@{}}
			\toprule
			& & \multicolumn{4}{c}{Number of evaluations of} & & time \\
			Instance & Method & \(f\) & \(\cons\) & \(\grad f\) & \(\grad^2_{xx} \Lag\) & \#fact. & (s) \\
			\midrule
			\multirow{4}{*}{\texttt{LSE\_Hard}} & Our Method & 83 & 78 & 30 & 29 & 114 & \textbf{20.3} \\
			 & ARC & \textbf{37} & 58 & \textbf{27} & \textbf{21} & 6017 & 27.1 \\
			 & AdaN & 63 & \textbf{42} & 67 & 36 & \textbf{61} & 51.5 \\
			 & IPOPT & 972 & 969 & 76 & 76 & 104 & 58.7 \\
			\addlinespace[0.35em]
			\multirow{4}{*}{\texttt{NP\_Adult}} & Our Method & 43 & 66 & 48 & 42 & 42 & 6.0 \\
			 & ARC & 89 & 171 & 83 & 75 & 16023 & 13.6 \\
			 & AdaN & 118 & 56 & 123 & 46 & 115 & 9.6 \\
			 & IPOPT & \textbf{15} & \textbf{15} & \textbf{15} & \textbf{14} & \textbf{17} & \textbf{5.6} \\
			\addlinespace[0.35em]
			\multirow{4}{*}{\texttt{RobustERM\_Spam}} & Our Method & 95 & 124 & 104 & 94 & 94 & 130.9 \\
			 & ARC & 57 & 111 & 57 & 43 & 6047 & 128.1 \\
			 & AdaN & 410 & 173 & 418 & 156 & 404 & 311.3 \\
			 & IPOPT & \textbf{17} & \textbf{17} & \textbf{17} & \textbf{16} & \textbf{16} & \textbf{76.7} \\
			\addlinespace[0.35em]
			\multirow{4}{*}{\texttt{RobustERM\_Phishing}} & Our Method & \textbf{5} & 35 & 27 & \textbf{4} & \textbf{4} & \textbf{0.2} \\
			 & ARC & 12 & 27 & 16 & 8 & 19 & 1.2 \\
			 & AdaN & 41 & 17 & 49 & 6 & 39 & 1.1 \\
			 & IPOPT & 6 & \textbf{6} & \textbf{6} & 5 & 5 & 0.7 \\
			\addlinespace[0.35em]
			\multirow{4}{*}{\texttt{RobustERM\_Education}} & Our Method & \textbf{19} & 56 & 27 & \textbf{18} & \textbf{18} & \textbf{2.5} \\
			 & ARC & 36 & 74 & 39 & 27 & 3929 & 6.6 \\
			 & AdaN & 144 & 64 & 151 & 52 & 141 & 12.4 \\
			 & IPOPT & 47 & \textbf{47} & \textbf{19} & \textbf{18} & \textbf{18} & 6.3 \\
			\bottomrule
		\end{tabular}%
	}
\end{table}
\else
\begin{table}[ht]
	\centering
	\caption{Comparison of the number of evaluations of the objective function, $f$, constraints, $\cons$, objective gradient, $\grad f$, and Hessian of the Lagrangian, $\grad^2_{xx}\Lag$, linear system factorizations, and wall-clock time across solvers.}
	\label{tab:results}
	
	\small
	\setlength{\tabcolsep}{4pt} 
	
	\begin{tabular}{@{}llrrrrrr@{}}
		\toprule
		& & \multicolumn{4}{c}{Number of evaluations of} & & time \\
		Instance & Method & $f$ & $\cons$ & $\grad f$ & $\grad^2_{xx} \Lag$ & \#fact. & (s) \\
		\midrule
		\multirow{4}{*}{\texttt{LSE\_Hard}} & Our Method & 83 & 78 & 30 & 29 & 114 & \textbf{20.3} \\
		& ARC & \textbf{37} & 58 & \textbf{27} & \textbf{21} & 6017 & 27.1 \\
		& AdaN & 63 & \textbf{42} & 67 & 36 & \textbf{61} & 51.5 \\
		& IPOPT & 972 & 969 & 76 & 76 & 104 & 58.7 \\
		\addlinespace[0.35em]
		\multirow{4}{*}{\texttt{NP\_Adult}} & Our Method & 43 & 66 & 48 & 42 & 42 & 6.0 \\
		& ARC & 89 & 171 & 83 & 75 & 16023 & 13.6 \\
		& AdaN & 118 & 56 & 123 & 46 & 115 & 9.6 \\
		& IPOPT & \textbf{15} & \textbf{15} & \textbf{15} & \textbf{14} & \textbf{17} & \textbf{5.6} \\
		\addlinespace[0.35em]
		\multirow{4}{*}{\texttt{RobustERM\_Spam}} & Our Method & 95 & 124 & 104 & 94 & 94 & 130.9 \\
		& ARC & 57 & 111 & 57 & 43 & 6047 & 128.1 \\
		& AdaN & 410 & 173 & 418 & 156 & 404 & 311.3 \\
		& IPOPT & \textbf{17} & \textbf{17} & \textbf{17} & \textbf{16} & \textbf{16} & \textbf{76.7} \\
		\addlinespace[0.35em]
		\multirow{4}{*}{\texttt{RobustERM\_Phishing}} & Our Method & \textbf{5} & 35 & 27 & \textbf{4} & \textbf{4} & \textbf{0.2} \\
		& ARC & 12 & 27 & 16 & 8 & 19 & 1.2 \\
		& AdaN & 41 & 17 & 49 & 6 & 39 & 1.1 \\
		& IPOPT & 6 & \textbf{6} & \textbf{6} & 5 & 5 & 0.7 \\
		\addlinespace[0.35em]
		\multirow{4}{*}{\texttt{RobustERM\_Education}} & Our Method & \textbf{19} & 56 & 27 & \textbf{18} & \textbf{18} & \textbf{2.5} \\
		& ARC & 36 & 74 & 39 & 27 & 3929 & 6.6 \\
		& AdaN & 144 & 64 & 151 & 52 & 141 & 12.4 \\
		& IPOPT & 47 & \textbf{47} & \textbf{19} & \textbf{18} & \textbf{18} & 6.3 \\
		\bottomrule
	\end{tabular}
\end{table}
\fi


\section{Discussion and future work}\label{sec:discussion}

We emphasize that the primary contribution of this work is theoretical:
we establish convergence guarantees for the proposed method under mild
assumptions, whereas IPOPT, while practically effective, does not come
with the same theoretical guarantees in the general convex constrained
setting.
The above experiments are therefore preliminary computational experiments
intended to demonstrate that the method is a step toward a viable practical method.
However, important issues remain:
\begin{enumerate}
\item Modification of the method to be pure primal-dual. Since pure primal-dual interior point methods 
dominate practical interior point 
implementations \cite{byrd2006knitro,wachter2006implementation}, this is likely to improve performance.
\item Handling infeasible starts. There are many problems where it is not easy to find a feasible starting solution,
or the feasible region lacks an interior. To handle such cases one could potentially integrate the method
with a penalty or shifted log barrier approach.
\item Handling nonconvexity. While this paper builds on \citet{hinder2024worst} which handles nonconvex
constraints but requires problem parameters as input, this work only covers the convex case
(although it does obtain better guarantees).
Extending the method and analysis to the nonconvex setting is an important topic for future research.
\end{enumerate}

Ultimately, we believe our work will pave the way for the development of practical log barrier methods for general 
nonconvex constrained optimization with strong worst-case guarantees.
In the long term, we hope this will help improve the robustness
and reliability of these nonconvex interior-point methods \cite{vanderbei1999loqo,wachter2006implementation}, 
which are widely used in practice \cite{biegler2009large}.

\section{Acknowledgments}
The authors were supported by AFOSR grant \#FA9550-23-1-0242. The authors would like to thank Lorenz Biegler for his helpful feedback on a draft of this paper. The authors would also like to thank Akif Khan for identifying errors in an early draft and for the helpful discussions. 

\clearpage

\appendix

\section{A method for satisfying Condition~\ref{condition:search-dir}}\label{appendix-satisfy-condition-phi-unconstrained}

In this section, we detail our approach to satisfying Condition~\ref{condition:search-dir}.
The key idea is to perform a bisection search to find the root function $\phi_k$ defined 
in~\Cref{eq:bisection-scheme}.
Before presenting our algorithm and its guarantees we first present \Cref{lemma:phi-facts} which identifies the useful
properties of $\phi_k$.


\begin{lemma}\label{lemma:phi-facts}
For $\delta > 0$, define
$\DirDelta := -(\grad^2 \barrier(x_k) + \delta \eye)^{-1} \grad \barrier(x_k) \text{ and } 
\phi_k(\delta) := \varepsilon_k^{-1} \delta \|\DirDelta\|_2$
with $\| \grad \barrier(x_k) \|_2 \ge \varepsilon_k > 0$. Let $\eta_1, \eta_2 \in (0, 1)$ and $\SufficientConstant \in (0,1)$.
For all $\delta, \delta' > 0$ such that $\delta \le \delta'$ we have $\phi_k(\delta) \le \phi_k(\delta')$ 
and $\delta \phi_k(\delta') \le \delta' \phi_k(\delta)$.
Moreover, if $\delta \ge  \frac{\eta_1}{1 - \eta_1} \| \grad^2 \barrier(x_k) \|_2$ then $\phi_k(\delta) \ge \eta_1$. 
If $\delta \le \eta_2^2 \SufficientConstant \varepsilon_k^2 (\barrier(x_k) - \barrier^\star)^{-1}$ then either 
\begin{flalign}\label{eq:general-delta-eta1-condition}
\barrier(x_k + \DirDelta) > \barrier(x_k) + \SufficientConstant \cdot \Mk(\DirDelta) \quad \textbf{or} 
\quad \phi_k(\delta) \le \eta_2.
\end{flalign}
\end{lemma}

\begin{proof}
Let $Q \Lambda Q^\top$ be the eigendecomposition of $\grad^2 \barrier(x_k)$, i.e., 
$\grad^2 \barrier(x_k) = Q \Lambda Q^\top$. 
Denote the corresponding eigenvalues, i.e., the diagonal entries of $\Lambda$ as 
$\lambda_1,\ldots,\lambda_{\NumVar}$ which are nonnegative by convexity of $\barrier$.
Let $\hat g = Q^\top \grad \barrier(x_k)$. Then, 
\begin{flalign}\label{eq:phi-formula}
\phi_k(\delta)^2 =  \varepsilon_k^{-2} \delta^2 \| \DirDelta \|_2^2 
= \varepsilon_k^{-2} \delta^2 \| (\Lambda + \delta \eye)^{-1} \hat{g} \|_2^2  
= \sum_{i=1}^{\NumVar} \left(\frac{\delta}{\varepsilon_k (\lambda_i+\delta)} \right)^2  [\hat{g}]_i^2 . 
\end{flalign}
Consider $\delta,\delta'$ such that $0 < \delta \le \delta'$, then for each $i$, we have
\[
\frac{\delta}{\lambda_i+\delta}
\le
\frac{\delta'}{\lambda_i+\delta'}
\le \frac{\delta'}{\lambda_i+\delta} = 
\frac{\delta'}{\delta} \cdot \frac{\delta}{\lambda_i+\delta}.
\]
The first inequality follows from monotonicity of
$\delta \mapsto \delta/(\lambda_i+\delta)$ because $\lambda_i \ge 0$.
The first inequality and \Cref{eq:phi-formula} also show that $\phi_k$ is nondecreasing.
Therefore, by \Cref{eq:phi-formula} we have
\[
\phi_k(\delta')^2 = \sum_{i=1}^{\NumVar} \left(\frac{\delta'}{\varepsilon_k (\lambda_i+\delta')} \right)^2  [\hat{g}]_i^2  
\le  \sum_{i=1}^{\NumVar} \left(\frac{\delta'}{\delta} \cdot \frac{\delta}{\varepsilon_k (\lambda_i+\delta)} \right)^2  [\hat{g}]_i^2 
= \left(\frac{\delta'}{\delta} \phi_k(\delta) \right)^2.
\]
Thus, $\delta \phi_k(\delta') \le \delta' \phi_k(\delta)$ as desired.

If $\delta \ge  \frac{\eta_1}{1 - \eta_1} \| \grad^2 \barrier(x_k) \|_2$ then by 
$(i)$ $\| \grad \barrier(x_k) \|_2 = 
\| (  \grad^2 \barrier(x_k) + \delta \eye) (\grad^2 \barrier(x_k) + \delta \eye)^{-1} \grad \barrier(x_k) \|_2 \\
\le (\delta + \| \grad^2 \barrier(x_k) \|_2) \| (\grad^2 \barrier(x_k) + \delta \eye)^{-1} \grad \barrier(x_k) \|_2$, 
$(ii)$ $\| \grad \barrier(x_k) \| \ge \varepsilon_k$, and $(iii)$ $\delta \ge  \frac{\eta_1}{1 - \eta_1} \| \grad^2 \barrier(x_k) \|_2$ we have 
\[
\phi_k(\delta) = \varepsilon_k^{-1} \delta \| (\grad^2 \barrier(x_k) + \delta \eye)^{-1} \grad \barrier(x_k) \|_2 
\underge{(i)} \varepsilon_k^{-1} \delta \frac{\| \grad \barrier(x_k) \|_2}{\delta + \| \grad^2 \barrier(x_k) \|_2} 
\underge{(ii)} \frac{\delta}{\delta + \| \grad^2 \barrier(x_k) \|_2} \underge{(iii)} \eta_1.
\]
If $\barrier(x_k + \DirDelta) \le \barrier(x_k) + \SufficientConstant \Mk(\DirDelta)$ then
\begin{flalign*}
\barrier(x_k) - \barrier^\star & \ge \barrier(x_k) - \barrier(x_k + \DirDelta) 
\ge -\SufficientConstant \Mk(\DirDelta) 
= -\SufficientConstant \grad \barrier(x_k)^\top \DirDelta - \frac{\SufficientConstant}{2}\DirDelta^\top \grad^2 \barrier(x_k)\DirDelta \\
&\undereq{(i)} \frac{\SufficientConstant}{2}\DirDelta^\top \grad^2 \barrier(x_k)\DirDelta + \SufficientConstant \delta \| \DirDelta \|^2 
\underge{(ii)} \SufficientConstant  \delta \| \DirDelta \|^2 
= \frac{\SufficientConstant \varepsilon_k^2}{\delta} \phi_k(\delta)^2.
\end{flalign*}
where $(i)$ uses \Cref{eq:search-direction} and $(ii)$ uses $\grad^2 \barrier(x_k) \succeq 0$.
If $\delta \le \eta_2^2 \SufficientConstant  \varepsilon_k^2 (\barrier(x_k) - \barrier^\star)^{-1}$ 
then rearranging the previous inequality gives 
$\phi_k(\delta) \le \sqrt{\delta \cdot \frac{\barrier(x_k) - \barrier^\star}{\SufficientConstant \varepsilon_k^2}} \le \eta_2$
as desired.
\end{proof}

\algrenewcommand\algorithmicrequire{\textbf{Input:}}
\algrenewcommand\algorithmicensure{\textbf{Output:}}

\begin{algorithm}[H]
\caption{Compute regularization parameter $\delta$ that satisfies Condition~\ref{condition:search-dir}}
\label{alg:compute-regularized-parameter}
\footnotesize
\begin{algorithmic}[1]
    \Require $\DeltaInit \in (0,\infty)$; $\delta_{\max} \in [\DeltaInit, \infty)$, $\eta_1 \in (0,1/2)$, $\eta_2 \in (\eta_1,1)$,  $\SufficientConstant \in (0,1)$
    \State $(\ell,u) \gets \Call{FindSearchInterval}{\DeltaInit, \delta_{\max}, \eta_1, \eta_2}$
    \State \Return $\Call{GeometricBisectionSearch}{\ell,u, \eta_1, \eta_2}$

    \Statex

\Function{FindSearchInterval}{$\DeltaInit, \delta_{\max}, \eta_1, \eta_2$}
		\vspace{0.1cm}
        \State $\SideA[1] \gets \DeltaInit$, $v_{\SideA[1]} \gets \phi_k(\DeltaInit)$, 
		$\SideB[1] \gets \DeltaInit$, $v_{\SideB[1]} \gets v_{\SideA[1]}$
        \State $r \gets \begin{cases}
            2, & v_{\SideA[1]} <\eta_1,\\
            1/2, & \text{otherwise},
        \end{cases}$
        \For{$i=1$ to $\infty$}
		\vspace{0.05cm}
			\If{$\eta_1 \le v_{\SideB}$ \textbf{ and  } [$v_{\SideB} \le \eta_2$ \textbf{~~or~~} 
			$\barrier(x_k + \DirDelta[\SideB]) > \barrier(x_k) + \SufficientConstant \Mk( \DirDelta[\SideB])$]}\label{alg-line:sufficient-decrease-fails-or-eta1-good-and-eta-2-good}
				\State \Return $(\SideB,\SideB)$
			\ElsIf{$\min\{v_{\SideA},v_{\SideB}\}\le\eta_1$  \textbf{and} 
			$\max\{v_{\SideA},v_{\SideB}\}\ge\eta_2$ }\label{alg-line:good-interval}
                \State \Return $(\min\{\SideA,\SideB\},\max\{\SideA,\SideB\})$
            \EndIf 
            \State $\SideA[i+1] \gets \SideB$, $v_{\SideA[i+1]} \gets v_{\SideB}$, 
			$\SideB[i+1] \gets \min\{\delta_{\max}, \DeltaInit r^{2^{(i-1)/2}} \}$, $v_{\SideB[i+1]} \gets \phi_k(\SideB[i+1])$
        \EndFor
    \EndFunction

    \Statex
    \Function{GeometricBisectionSearch}{$\ell_1,u_1, \eta_1,\eta_2$}
		\If{$\ell_1=u_1$}
			\State \Return $\ell_1$
		\EndIf
        \For{$j=1$ to $\infty$}
            \State $\DeltaMid \gets \sqrt{\ell_j u_j}$, $v_m \gets \phi_k(\DeltaMid)$
            \If{$\eta_1 \le v_m \le \eta_2$}
                \State \Return $\DeltaMid$
            \ElsIf{$v_m<\eta_1$}
                \State $\ell_{j+1} \gets \DeltaMid$, $u_{j+1} \gets u_j$
            \Else
                \State $\ell_{j+1} \gets \ell_j$, $u_{j+1} \gets \DeltaMid$
            \EndIf
        \EndFor
    \EndFunction
\end{algorithmic}
\end{algorithm}

\begin{proposition}
Let $\| \grad \barrier(x_k) \|_2 \ge \varepsilon_k > 0$ and $\delta_{\max} \ge \frac{\eta_1}{1-\eta_1}\| \grad^2 \barrier(x_k) \|_2$.
Then, the total number of evaluations of $\phi_k$ used by \Cref{alg:compute-regularized-parameter} is at most
\[
6
+3 \log_2\log_2 N_k
+
\log_2\!\left(
1 + \frac{1}{\log_2(\eta_2^2/\eta_1^2)} \right) ~\text{ where }~ N_k :=
\max\left\{
2,
\frac{\eta_1 \| \grad^2 \barrier(x_k) \|_2}
{(1-\eta_1)\DeltaInit},
\frac{\DeltaInit(\barrier(x_k) - \barrier^\star)}
{\eta_2^2 \SufficientConstant \varepsilon_k^2}
\right\}.
\]
\end{proposition}

\begin{proof}
We begin by analyzing the runtime of \textbf{FindSearchInterval}. 
Let $i_\star$ be the number of iterations used by its loop, let $A_k:=\frac{\eta_1}{1-\eta_1}\| \grad^2 \barrier(x_k) \|_2$, 
and let $B_k:=\eta_2^2 \SufficientConstant \varepsilon_k^2 (\barrier(x_k)-\barrier^\star)^{-1}$.
Consider the case that $\phi_k(\DeltaInit) = v_{\SideA[1]} \in [\eta_1,\eta_2]$, then 
by Line~\ref{alg-line:sufficient-decrease-fails-or-eta1-good-and-eta-2-good} we have $i_\star=1$. 
Consider the case that $\phi_k(\DeltaInit) = v_{\SideA[1]} < \eta_1$. 
If $v_{\SideA} \le \eta_1$ and $\eta_1 \le v_{\SideB}$, 
then by Line~\ref{alg-line:sufficient-decrease-fails-or-eta1-good-and-eta-2-good} or Line~\ref{alg-line:good-interval}
the for loop ends. 
Thus since $\phi_k(\DeltaInit) = v_{\SideA[1]} < \eta_1$ the first index $i$ with $\eta_1 \le v_{\SideB}$
must also have $v_{\SideA} < \eta_1$, triggering the end of the loop; by \Cref{lemma:phi-facts} this occurs no later than the first index with $\SideB \ge A_k$. 
Since $\SideB[i-1]= \DeltaInit 2^{2^{(i-3)/2}} < A_k$ for $i \in [3,i_\star]$ and $N_k\ge\max\{2,A_k/\DeltaInit\}$, we get
$i_\star < 3 + 2 \log_2 \log_2 N_k$.
On the other hand, consider the case that $\phi_k(\DeltaInit) = v_{\SideA[1]} > \eta_2$.
If $v_{\SideB} \le \eta_2$ and $\eta_2 \le v_{\SideA}$, or $\eta_1 \le v_{\SideB}$ and 
$\barrier(x_k + \DirDelta[\SideB]) > \barrier(x_k) + \SufficientConstant \Mk( \DirDelta[\SideB])$, 
then by Line~\ref{alg-line:sufficient-decrease-fails-or-eta1-good-and-eta-2-good} or Line~\ref{alg-line:good-interval}
the for loop ends. 
Since $v_{\SideA[1]} > \eta_2$, the first index with $v_{\SideB} \le \eta_2$ or 
$\barrier(x_k + \DirDelta[\SideB]) > \barrier(x_k) + \SufficientConstant \Mk( \DirDelta[\SideB])$
will trigger the for loop to end;
by \Cref{lemma:phi-facts}, this occurs no later than the first index with 
$\SideB\le B_k$. Since $\SideB[i-1]=\DeltaInit 2^{-2^{(i-3)/2}} > B_k$ for $i \in [3,i_\star]$ and $N_k\ge\max\{2,\DeltaInit/B_k\}$, this also 
gives $i_\star < 3 + 2 \log_2 \log_2 N_k$.

Next, we analyze \textbf{GeometricBisectionSearch}. Let $j_\star$ be the number of iterations used by 
its loop with $j_\star = 0$ denoting the case that $\ell_1 = u_1$ and consider the nontrivial case that $j_\star > 1$.
Since $\ell_1<u_1$, by construction and monotonicity of $\phi_k$, we have $\phi_k(\ell_1)\le\eta_1$ and 
$\phi_k(u_1)\ge\eta_2$, and these inequalities are maintained until termination. Observe that, if 
$u_j / \ell_j \le \eta_2^2 / \eta_1^2$ then by \Cref{lemma:phi-facts} we have 
\[
\eta_1 \le \eta_2 \frac{\ell_j}{\sqrt{\ell_j u_j}} \le \frac{\ell_j}{\sqrt{\ell_j u_j}} \phi_k(u_j) \le 
\phi_k(\sqrt{\ell_j u_j}) \le \frac{u_j}{\sqrt{\ell_j u_j}} \phi_k(\ell_j) \le 
\eta_1 \frac{u_j}{\sqrt{\ell_j u_j}} \le \eta_2
\]
which implies \textbf{GeometricBisectionSearch} terminates at this iteration with $\DeltaMid = \sqrt{\ell_j u_j}$, i.e., 
$j = j_\star$. Thus, for $j \in [2,j_\star]$ we have 
$\log_2(\eta_2^2 / \eta_1^2) < \log_2(u_{j-1} / \ell_{j-1}) = 2^{2-j} \log_2(u_1 / \ell_1)$ where the equality holds 
from induction on $\log_2(u_j/\ell_j) = 2^{-1} \log_2(u_{j-1}/\ell_{j-1})$.
Rearranging this inequality with $j = j_\star$, it follows that 
\[
j_\star<2+\log_2\!\left(\frac{\log_2(u_1/\ell_1)}{\log_2(\eta_2^2/\eta_1^2)}\right) \underle{(i)} 
2+\log_2\!\left(\frac{\max\{2^{(i_\star - 2) / 2}, 1\}}{\log_2(\eta_2^2/\eta_1^2)}\right) 
= \max\left\{ 1 + \frac{i_\star}{2}, 2 \right\} + \log_2\!\left(\frac{1}{\log_2(\eta_2^2/\eta_1^2)}\right)
\]
where $(i)$ uses that $u_1 / \ell_1 \le \max\{ 2^{2^{(i_\star - 2) / 2}}, 1\}$
from definition of \textbf{FindSearchInterval}.
Since \textbf{FindSearchInterval} evaluates $\phi_k$ exactly $i_\star$ 
times and \textbf{GeometricBisectionSearch} evaluates $\phi_k$ exactly $j_\star$ times, 
adding their bounds together gives the desired result.
\end{proof}

\begin{remark}[Actual implementation of the search direction]\label{rem:actual-search-direction-implementation}
\Cref{alg:compute-regularized-parameter} does not specify how to choose $\DeltaInit$.
We first check if
\Cref{eq:direction-condition-eta1} holds with $\delta_k = 0$, if so we use $\delta_k = 0$, 
otherwise we set
\[ 
\DeltaInit \gets \begin{cases}
\delta_{\max} \ColdStartScale / \sqrt{\NumVar} & \delta_{k-1} = 0 \text{ or } k = 0 \\
\min\{ \delta_{k-1}, \delta_{\max} \} & \delta_{k-1} \neq 0, ~\alpha_{k-1} < 1 \\ 
\min\{ \RegularizationShrink \delta_{k-1}, \delta_{\max} \} & \delta_{k-1} \neq 0, ~\alpha_{k-1} = 1
\end{cases}
\]
where $\delta_{\max} = \frac{\eta_1}{1 - \eta_1} \| \grad^2 \barrier(x) \|_F$, $\RegularizationShrink \in (0,1)$, and $\ColdStartScale \in (0,\infty)$.
We then call \Cref{alg:compute-regularized-parameter} with this value of $\DeltaInit$ and $\delta_{\max}$.
\end{remark}

\section{A line search routine for satisfying Condition~\ref{step-size-condition}}\label{appendix-line-search-routine}
\newcommand{\alphaInit}{\hat{\alpha}_{k}}

\Cref{alg:compute-step-size} provides a standard line search routine for computing the step size. 
The routine starts from an initial guess of the step size $\alphaInit$ and either 
increases the step size (i.e., forward tracks) or decreases the step size (i.e., backtracks) in
search of a step size that satisfies Condition~\ref{step-size-condition}.
\Cref{prop:line-search-iteration-bound} provides guarantees on 
the number of iterations this routine requires. 
\Cref{prop:line-search-iteration-bound} specifies that $\alphaInit = \min\{\StartingTargetReduction/(-\mk), 1\}$ where $\StartingTargetReduction \in (0,\infty)$.
It is natural to select $\alphaInit = 1$, in which case the line search 
is pure backtracking but this can be problematic, for example, if $\grad^2 \barrier(x_k) = 0$
and $\delta_k$ is chosen to be very small, then it is possible that only an exceptionally
small step size satisfies the sufficient decrease conditions, which could take a backtracking method
a long time to find.

\newcommand{\lsIter}{j}

\begin{algorithm}[H]
\caption{Forward/back tracking method to compute step size satisfying Condition~\ref{step-size-condition}.}
\label{alg:compute-step-size}
\footnotesize
\begin{algorithmic}[1]
    \Require $\alphaInit \in (0,1]$; $\SufficientConstant \in (0,1)$; $\backtrackingFactor \in (0,1)$; $x_k$; $\dir{x}_k$
    \Function{SufficientDecreaseConditionHolds}{$\alpha$}
        \State \Return $\barrier(x_k + \alpha \dir{x}_k) \leq \barrier(x_k) + \SufficientConstant \Mk(\alpha \dir{x}_k)$
    \EndFunction
    \State $\lsIter \gets 0$, $\alpha_{k,\lsIter} \gets \alphaInit$
    \If{\textbf{not} \Call{SufficientDecreaseConditionHolds}{$\alpha_{k,\lsIter}$}} \textit{backtrack, i.e.,}
        \While{\textbf{not} \Call{SufficientDecreaseConditionHolds}{$\alpha_{k,\lsIter}$}}\label{line:backtracking-branch}
            \State $\alpha_{k,\lsIter+1} \gets \backtrackingFactor \alpha_{k,\lsIter}$, $\lsIter \gets \lsIter+1$
        \EndWhile
    \Else \textit{~~forward track, i.e.,}
        \While{$\alpha_{k,\lsIter} < 1$ \textbf{and} \Call{SufficientDecreaseConditionHolds}{$\min\{1,\alpha_{k,\lsIter}/\backtrackingFactor\}$}}\label{line:forwardtracking-branch}
            \State $\alpha_{k,\lsIter+1} \gets \min\{1,\alpha_{k,\lsIter}/\backtrackingFactor\}$, $\lsIter \gets \lsIter+1$\label{line:forwardtracking-update}
        \EndWhile
    \EndIf
    \State \Return $\alpha_{k,\lsIter}$
\end{algorithmic}
\end{algorithm}

\begin{fact}[Well-known]\label{eq:gradient-dot-product-mk-relationship}
If Condition~\ref{condition:search-dir} holds then 
$-\grad \barrier(x_k)^\top \dir{x}_k \le -2 \mk$.
\end{fact}

\begin{proof}
Substituting $\grad^2 \barrier(x_k) \dir{x}_k  = - \grad \barrier(x_k) - \delta_k \dir{x}_k$ from 
\Cref{eq:search-direction} into $\mk = \grad \barrier(x_k)^\top \dir{x}_k + \frac{1}{2} (\dir{x}_k)^\top \grad^2 \barrier(x_k) \dir{x}_k$
gives $\mk = \frac{1}{2} \grad \barrier(x_k)^\top \dir{x}_k - \frac{\delta_k}{2} \| \dir{x}_k \|_2^2$. 
Using $\delta_k \ge 0$ gives the desired result.
\end{proof}

\begin{proposition}\label{prop:line-search-iteration-bound}
Suppose $\dir{x}_k$ satisfies Condition~\ref{condition:search-dir}.
If $\grad \barrier(x_k)^\top \dir{x}_k < 0$, and $\alphaInit = \min\{\StartingTargetReduction/(-\mk), 1\}$ then \Cref{alg:compute-step-size} terminates 
with a step size $\alpha_k=\alpha_{k,\lsIter}$ satisfying, Condition~\ref{step-size-condition}. 
Moreover, if backtracking branch (Line~\ref{line:backtracking-branch}) is entered then 
$j \le \log_{1/\backtrackingFactor} \frac{2\StartingTargetReduction}{\barrier(x_k)-\barrier(x_{k+1})}$
and if the forward tracking branch (Line~\ref{line:forwardtracking-branch}) is entered then
$j \le \max\left\{0, 1 + \log_{1/\backtrackingFactor} \frac{2 (\barrier(x_k)-\barrier^\star)}{\SufficientConstant \StartingTargetReduction} \right\}$.
\end{proposition}

\begin{proof}
Fix $k$, and let $j$ be the value of $\lsIter$ returned by
Algorithm~\ref{alg:compute-step-size}. Since $\grad \barrier(x_k)^\top \dir{x}_k<0$
and $\SufficientConstant < 1$, we conclude by Taylor's theorem that for sufficiently small $\alpha$  the sufficient decrease 
condition holds. Moreover, by construction \Cref{alg:compute-step-size} terminates 
with a point satisfying Condition~\ref{step-size-condition}.

It remains to bound $j$. First, consider the case that the backtracking branch is entered (Line~\ref{line:backtracking-branch}).
Then, by $(i)$ convexity of $\barrier$, $(ii)$ definition of $\alphaInit$, and $(iii)$ 
\Cref{eq:gradient-dot-product-mk-relationship},
\begin{flalign*}
\barrier(x_k)-\barrier(x_{k+1}) &\underle{(i)}
    -\alpha_{k,j}\grad \barrier(x_k)^\top \dir{x}_k  =
    -\backtrackingFactor^j
    \alphaInit
    \grad \barrier(x_k)^\top \dir{x}_k  \underle{(ii)}
     \frac{\backtrackingFactor^j \StartingTargetReduction \grad \barrier(x_k)^\top \dir{x}_k}{\mk} \underle{(iii)} 2 \backtrackingFactor^j \StartingTargetReduction.
\end{flalign*}
Rearranging yields $j \le \log_{1/\backtrackingFactor} \frac{2\StartingTargetReduction}{\barrier(x_k)-\barrier(x_{k+1})}$ as desired.

Next,  consider the case that the forward tracking branch is entered (Line~\ref{line:forwardtracking-branch}).
If $j = 0$ then the result trivially holds.
On the other hand, if $j > 0$ then $\alphaInit = \StartingTargetReduction / (-\mk)$ and
\begin{flalign*}
\barrier(x_k)-\barrier(x_{k} + \alpha_{k,j-1}\dir{x}_k) &\underge{(i)}  -\SufficientConstant \Mk(\alpha_{k,j-1}\dir{x}_k) 
\underge{(ii)} -\frac{\SufficientConstant  \alpha_{k,j-1} \mk}{2}
\undereq{(iii)}  -\frac{\SufficientConstant \mk}{2} \cdot \frac{\backtrackingFactor^{1-j} \StartingTargetReduction}{-\mk}
= \frac{\SufficientConstant \StartingTargetReduction \backtrackingFactor^{1-j}}{2}.
\end{flalign*}
where $(i)$ holds by Line~\ref{line:forwardtracking-branch}, Line~\ref{line:forwardtracking-update} and 
the definition of \Call{SufficientDecreaseConditionHolds}{}, $(ii)$
\Cref{lem:model-alpha-k-bound} and $(iii)$ $\alphaInit = \StartingTargetReduction / (-\mk)$ and $\alpha_{k,j-1} = \alphaInit \backtrackingFactor^{1-j}$.
Rearranging and using $\barrier(x_{k} + \alpha_{k,j-1}\dir{x}_k) \ge \barrier^\star$ gives the desired result.
\end{proof}

\begin{remark}
The $\log_{1/\backtrackingFactor} \frac{2\StartingTargetReduction}{\barrier(x_k)-\barrier(x_{k+1})}$ term in
the bound on $j$ from \Cref{prop:line-search-iteration-bound} could be large if $\barrier(x_k)-\barrier(x_{k+1})$ 
was very tiny. 
However, we can employ the lower bounds on $\barrier(x_k)-\barrier(x_{k+1})$ 
from the body of the paper to rule out this possibility. 
If $\min_{i \in [7]} \alphaInd{k}^{(i)} < 1$ then 
\Cref{lemBarrierProgress} shows that $\barrier(x_k)-\barrier(x_{k+1})$ is bounded below
proportional to $\mu^{5/3}$. On the other hand, if $\min_{i \in [7]} \alphaInd{k}^{(i)} \ge 1$
then $\alpha_k = 1$ by \Cref{lemMinimumStepSizeIPM} which implies that the 
forward tracking branch was entered and thus 
$j \le \max\left\{0, 1 + \log_{1/\backtrackingFactor} \frac{2 (\barrier(x_k)-\barrier^\star)}{\SufficientConstant \StartingTargetReduction} \right\}$.
\end{remark}


\clearpage

\bibliography{references}

\end{document}